\documentclass[12pt]{amsart}
\usepackage{amsmath}
\usepackage{amsthm}
\usepackage{amssymb}
\usepackage{amscd}
\usepackage{amsfonts}
\usepackage{amsbsy}
\usepackage{epsfig,afterpage}
\usepackage[dvips]{psfrag}
\usepackage{subfigure}
\usepackage{epsfig}
\usepackage{amsmath}
\usepackage{srcltx}
\usepackage{amsfonts}
\usepackage{amssymb}
\usepackage{psfrag}
\usepackage{verbatim}
\usepackage{graphicx}
\usepackage{amsmath}
\usepackage{amsthm}
\usepackage{amssymb}
\usepackage{amscd}
\usepackage{amsfonts}
\usepackage{amsbsy}
\usepackage{epsfig}
\usepackage[dvips]{psfrag}
\usepackage{multirow}
\usepackage{graphics}
\usepackage{epstopdf}
\usepackage{overpic}
\usepackage{float}

\usepackage{graphicx}

\newtheorem {theorem} {Theorem}%[section]
\newtheorem {proposition} [theorem]{Proposition}

\newtheorem {lemma}  [theorem]{Lemma}
\newtheorem {example} [theorem]{Example}

\newtheorem {definition} [theorem]{Definition}

\newtheorem{mtheorem}{Theorem}

\usepackage{float}
\usepackage{tikz, wrapfig}
\tikzset{node distance=3cm, auto}
\allowdisplaybreaks

\begin{document}

\title[Classification of constrained systems]
{On the classification of planar constrained differential systems under topological equivalence}

\author[O. H. Perez and P. R. da Silva]
{Otavio Henrique Perez and Paulo Ricardo da Silva}

\address{S\~{a}o Paulo State University (Unesp), Institute of Biosciences, Humanities and Exact Sciences. Rua C. Colombo, 2265, CEP 15054--000. S. J. Rio Preto, S\~ao Paulo, Brazil.}

\email{oh.perez@unesp.br}
\email{paulo.r.silva@unesp.br}

\thanks{ .}


\date{}
\dedicatory{}
\maketitle

\begin{abstract}
This paper concerns the local study of analytic constrained differential systems (or impasse systems) of the form $A(x)\dot{x} = F(x)$, $x\in\mathbb{R}^{2}$, where $F$ is a vector field and $A$ is a matrix valued function. Using techniques of resolution of singularities with weighted blow ups, we extend well-known results in the literature on topological classification of phase portraits of planar vector fields to the context of such class of systems. We also study and classify phase portraits of constrained systems near singular points of the so called impasse set $\Delta=\{x:\det A (x) = 0\}$.
\end{abstract}

\section{Introduction, statement of the problem and main results}

\noindent

A \textbf{constrained differential system} (or simply constrained system) defined in an open set $U\subseteq\mathbb{R}^{2}$ is given by
\begin{equation}\label{eq-def-constrained}
A(\mathbf{x})\dot{\mathbf{x}} = F(\mathbf{x}),
\end{equation}
where $\mathbf{x}\in U\subseteq\mathbb{R}^{2}$, $A(\mathbf{x})$ is a matrix valued function and $F:U\rightarrow \mathbb{R}^{2}$ is a vector field. We assume that these maps are analytic in the open set $U$. A point $\mathbf{x_{0}}$ such that $\delta(\mathbf{x_{0}}) = 0$ is called \textbf{impasse point}, where $\delta:U\rightarrow\mathbb{R}$ denotes the determinant function $\det A(\mathbf{x})$. The so called \textbf{impasse set} is the set $\Delta = \{\mathbf{x}\in U: \delta(\mathbf{x}) = 0\}$.

Systems of the form \eqref{eq-def-constrained} have been widely studied in the literature. We refer to \cite{CardinSilvaTeixeira,ChuaOka1,LopesSilvaTeixeira,RabierRheinboldt,SotoZhito,Zhito} for theoretical aspects on this subject. For applications in electric circuit theory, see \cite{Riaza, Smale}.

This paper deals with constrained systems such that the matrix $A$ does not have constant rank. We will further suppose that $\delta(\mathbf{x}) = 0$ only in a closed subset of $U$. This implies that the adjoint matrix $A^{*}$ of $A$ is not identically zero and $AA^{*} = A^{*}A = \det(A)I$.

Multiplying both sides of \eqref{eq-def-constrained} by $A^{*}$, it can be shown that $\gamma$ is solution of \eqref{eq-def-constrained} if, and only if, $\gamma$ is a solution of
\begin{equation}\label{eq-def-constrained-2}
\delta(\mathbf{x})\dot{\mathbf{x}} = A^{*}(\mathbf{x})F(\mathbf{x}).
\end{equation}

The constrained system written in the form \eqref{eq-def-constrained-2} will be called \textbf{diagonalized constrained system} and, just as in \cite{CardinSilvaTeixeira}, the vector field
\begin{equation}\label{eq-def-adjoint-vector field}
X(\mathbf{x}) = A^{*}(\mathbf{x})F(\mathbf{x})
\end{equation}
will be called \textbf{adjoint vector field}. Observe that, outside the impasse set, the constrained system \eqref{eq-def-constrained} can be written as
$$\dot{\mathbf{x}} = A^{-1}(\mathbf{x})F(\mathbf{x}),$$
which is a classical ordinary differential equation.

At an impasse point $\mathbf{x_{0}}$, since the matrix $A(\mathbf{x_{0}})$ is not invertible, the classical results on the existence and uniqueness of the solutions break down. However, near such point we can describe the phase portrait as follows. In the open set where $\delta(\mathbf{x}) > 0$, the phase portrait is given by the adjoint vector field $X$. On the other hand, in the open set where $\delta (\mathbf{x}) < 0$ the phase portrait is given by $-X$.

The study of singularities of constrained systems whose matrix $A$ does not have constant rank was addressed in \cite{SotoZhito,Zhito}, where the authors considered singularities near regular points of the impasse set $\Delta$, such as tangencies between $\Delta$ and the adjoint vector field $X$ and equilibrium points of $X$ on $\Delta$. By regular impasse point we mean that the determinant function $\delta$ satisfies $\nabla\delta(\mathbf{x})\neq 0$. In such papers was adopted the following notion of equivalence.

\begin{definition}\label{def-orbitally-eq}
Two constrained systems $\delta_{i}(\mathbf{x})\dot{\mathbf{x}} = X_{i}(\mathbf{x})$, $i = 1,2,$  are \textbf{$C^{k}$-orbitally equivalent at the points $p_{i}\in\Delta_{i}$}, $i = 1,2$, if there are two open sets $V_{1}\ni p_{1}$ and $V_{2}\ni p_{2}$ and a $C^{k}$ diffeomorphism $\Phi: V_{1}\rightarrow V_{2}$ such that
\begin{enumerate}
  \item $\Phi$ maps the impasse set $\Delta_{1}$ to the impasse set $\Delta_{2}$;
  \item $\Phi$ maps the phase curves of $X_{1}$ in $V_{1}\backslash\Delta_{1}$ to the phase curves of $X_{2}$ in $V_{2}\backslash\Delta_{2}$, not necessarily preserving orientation.
\end{enumerate}
If $\Phi$ is a homeomorphism, we say that is a \textbf{$C^{0}$-orbital equivalence} or \textbf{topological orbital equivalence}. In the case where $\Phi$ preserves orientation, we say that $\Phi$ is an \textbf{orientation preserving $C^{k}$-orbital equivalence}.
\end{definition}

For normal forms of planar constrained systems where the matrix $A$ has constant rank, see \cite{ChuaOka1} for instance.

The main goal of this paper is to study singularities and classify phase portraits of constrained differential systems whose matrix $A$ does not have constant rank. Differently from \cite{SotoZhito,Zhito}, we consider singular points of the impasse set $\Delta$. Furthermore, the adjoint vector field $X$ can present equilibrium points on $\Delta$ more degenerated than those ones considered in these references. More precisely:

\begin{definition}\label{def-singularities}
We say that $p\in \Delta$ is a \textbf{singularity of the constrained system \eqref{eq-def-constrained-2}} if one of the following conditions is satisfied:
\begin{enumerate}
  \item $p$ is an equilibrium point of the adjoint vector field \eqref{eq-def-adjoint-vector field}, that is, $X(p) = 0$;
  \item the impasse set $\Delta$ is not smooth at $p$;
  \item $p$ is a tangency point between $X$ and $\Delta$, that is, $d\delta (X) = 0$.
\end{enumerate}
Otherwise, we say that $p$ is \textbf{non-singular}. See Figure \ref{fig-def-sing}.
\end{definition}

\begin{figure}[h!]
  \center{\includegraphics[width=0.75\textwidth]{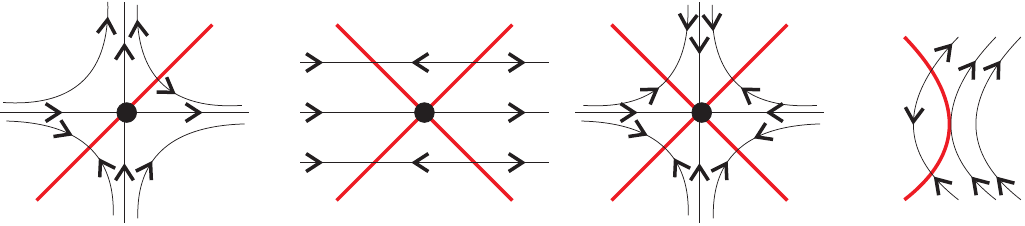}\hspace{0.5cm}\includegraphics[width=0.15\textwidth]{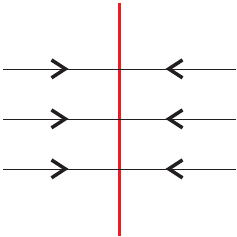}}\\
  \caption{Singularities of a planar constrained system. The rightmost phase portrait represents a non-singular point. The impasse set is highlighted in red.}\label{fig-def-sing}
\end{figure}

For our purposes, the resolution of singularities using weighted blow ups is the main tool. Resolution of singularities of vector fields was considered in \cite{BrunellaMiari, Dumortier, Panazzolo, Pelletier} and we refer to \cite{AlvarezFerragutJarque, DumortierLlibreArtes} for an introduction on this subject. The resolution of singularities of planar analytic constrained systems having impasse set was addressed in \cite{CardinSilvaTeixeira}, where the authors studied constrained systems whose impasse set is given by an homogeneous polynomial. In this paper are considered constrained systems whose impasse set is defined by an analytic function.

Roughly speaking, in the process of resolution we replace a singularity of the constrained system by a compact set homeomorphic to $\mathbb{S}^{1}$, called \textbf{exceptional divisor} (that is, we blow-up a singularity). In this new compact set, the constrained system (possibly) presents less degenerated singularities. Continuing this reasoning at each singularity that appears in the divisor, in the end of the process there will be only \textbf{elementary singularities} in the divisor. By elementary singularity, we mean the following.

\begin{definition}\label{def-elementary-points}
The planar constrained differential system \eqref{eq-def-constrained-2} is \textbf{elementary at $p$} if one of the following conditions holds:
\begin{enumerate}
  \item $p$ is a non singular point;
  \item $p$ is a semi-hyperbolic equilibrium point of $X$ (that is, at least one of its eigenvalues is nonzero) and $p\not\in\Delta$;
  \item If $p$ is a semi-hyperbolic equilibrium point of $X$ and $p\in\Delta$, then $\Delta$ coincides with a local separatrix of $X$ at $p$.
\end{enumerate}
See Figure \ref{fig-def-elem}.
\end{definition}

\begin{figure}[h!]
  \center{\includegraphics[width=0.25\textwidth]{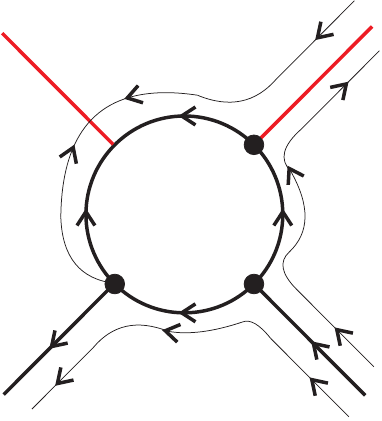}}\\
  \caption{Elementary points of a constrained system on the divisor. The impasse set is highlighted in red.}\label{fig-def-elem}
\end{figure}

Once there are only elementary singularities, we study the dynamics on the exceptional divisor in order to classify the phase portrait of the original constrained system near the blown-up singularity.

The resolution of singularities was applied in the study of topological classification of phase portraits of vector fields (that is, classification of phase portraits under topological equivalence) in \cite{BrunellaMiari,Dumortier}, and the results presented here can be seen as a generalization for planar constrained systems of those ones in these papers.

All along the paper a constrained system defined near a point $p$ will be denoted by a triple $(X, \delta,p)$, where $X$ is the adjoint vector field \eqref{eq-def-adjoint-vector field} and $\delta$ is an analytic irreducible function such that the impasse set is given by $\Delta = \{\delta = 0\}$. When $p$ is the origin, a constrained system is simply denoted by $(X, \delta)$. After a suitable finite sequence of weighted blow ups, we obtain the so called \textbf{strict transformed $(\widetilde{X}, \widetilde{\delta},D)$ of} $(X, \delta,p)$, where $D$ denotes the excepcional divisor.

In what follows we summarize the main results of the paper.

Firstly we give a condition based on the resolution of singularities to establish when two constrained differential systems are topologically orbitally equivalents (or $C^{0}$-orbitally equivalents) near a singularity. This condition is the so called elementary singularity scheme (see Section \ref{sec-construction} for a precise definition). Roughly speaking, we prove that if two planar constrained systems have the same resolution of singularities, then they are equivalents. This result extends Theorem B of \cite{Dumortier} to the context of constrained systems. More precisely:

\begin{mtheorem}\label{coroo-eq-conjugacy}
Consider two constrained systems $(X, \delta, p_{1})$ and $(Y, \gamma, p_{2})$, where $p_{1}\in \Delta_{\delta}$ and $p_{2}\in\Delta_{\gamma}$ are singularities. If the strict transformed of both constrained systems have the same non-degenerated elementary singularity scheme, then $(X, \delta, p_{1})$ and $(Y, \gamma, p_{2})$ are orientation preserving orbitally $C^{0}$-equivalents.
\end{mtheorem}

The Theorem \ref{coroo-eq-conjugacy} is a key theorem that we will use to prove the Theorems \ref{teo-singular-curves-ade-resolution} and \ref{teo-brunellamiari-2}. The next goal is to classify phase portraits of planar constrained systems having singular impasse curve. We define the set $\mathcal{A}$ as the set of analytic planar constrained systems $(X, \delta)$ defined near the origin such that:
\begin{enumerate}
  \item The adjoint vector field is constant, that is, it satisfies $X(x,y) = (a,b)$, $a,b\in\mathbb{R}$ and $a^{2} + b^{2} \neq 0$.
  \item The function $\delta$ is one of the maps given by Arnold's ADE-classification \cite{Arnold} (see Table \ref{tab-arnold-ade} in Section \ref{sec-arnold}).
\end{enumerate}

The Theorem \ref{teo-singular-curves-ade-resolution} give conditions to assure when an analytic constrained system $(X, \delta)$ is orbitally equivalent to a system in $\mathcal{A}$, where the coefficients $a_{-1,0}$ and $b_{0,-1}$ in the expansion
\begin{equation}\label{eq-constrained-singular-curve-theorem-intro}
\left\{
  \begin{array}{rcl}
    \delta(x,y)\dot{x} & = & a_{-1,0} + a_{0,0}x + a_{-1,1}y + ..., \\
    \delta(x,y)\dot{y} & = & b_{0,-1} + b_{0,0}y + b_{1,-1}x + ...,
  \end{array}
\right.
\end{equation}
satisfy $a_{-1,0}^{2} + b_{0,-1}^{2} \neq 0$ and the function $\delta$ is one of the functions in Table \ref{tab-arnold-ade}.

\begin{mtheorem}\label{teo-singular-curves-ade-resolution}
Consider an analytic planar constrained system $(X, \delta)$ near the origin such that $X(0) \neq 0$ and the origin is an ADE-type singular point of $\delta$. Then the system \eqref{eq-constrained-singular-curve-theorem-intro} is orientation preserving $C^{0}$-orbitally equivalent to a system in $\mathcal{A}$ if it satisfies one of the conditions in Table \ref{tab-conditions-teo-c}, Section \ref{sec-arnold}.
\end{mtheorem}

In section \ref{sec-arnold} we present all the possible phase portraits of the systems that belong to $\mathcal{A}$.

Concerning Theorem \ref{teo-brunellamiari-2}, suppose that the origin $0$ is a singularity of an analytic constrained system $(X,\delta)$, where $0$ is an equilibrium point of the adjoint vector field $X$ and a regular point of the impasse set $\Delta$ simultaneously. We associate a Newton polygon $\mathcal{P}$ to such constrained system and we show in Theorem \ref{teo-brunellamiari-2} that the terms in the boundary $\partial\mathcal{P}$ of the polygon determine the phase portrait near the origin under topological orbital equivalence. In other words, the terms in the boundary $\partial\mathcal{P}$ define the so called principal part $(X_{\Omega},\delta)$ of the analytic constrained system (see Section \ref{sec-newton-pol} for a precise definition), and we prove that, under some non-degeneracy conditions, the original constrained system is orbitally equivalent to its principal part.

From a practical point of view, Theorem \ref{teo-brunellamiari-2} says that if we want to study topological properties of a singularity in the regular part of $\Delta$, one can ``discard'' the terms of the Taylor expansion that are not in the boundary $\partial\mathcal{P}$ of the polygon. This result extends Theorem A of \cite{BrunellaMiari} to the context of constrained systems with regular impasse set.

\begin{mtheorem}\label{teo-brunellamiari-2}
Let $(X,\delta)$ be an analytic planar constrained system such that $\Delta = \{\delta = 0\}$ is a regular curve and the origin is an isolated Newton non-degenerated singularity. If $X$ has characteristic orbit (that is, the origin is neither a center nor a focus of $X$), then $(X_{\Omega},\delta)$ and $(X,\delta)$ are orientation preserving $C^{0}$-orbitally equivalent.
\end{mtheorem}

The paper is organized as follows. In Section \ref{sec-preliminaries} we briefly introduce weighted blowing ups and the Theorem of resolution of singularities for planar constrained vector fields. In Section \ref{sec-construction} is given the proof of Theorem \ref{coroo-eq-conjugacy}, where we use the resolution of singularities to establish when two constrained systems are $C^{0}$-orbitally equivalents. We apply this result in the following sections. In Section \ref{sec-arnold} we recall the construction of the Newton polygon and see how to define such mathematical object for planar constrained systems. Moreover, we study and classify phase portraits of constrained systems whose impasse curve has a singular point given by Arnold's ADE-classification. Finally, in Section \ref{sec-newton-pol} we show that the terms in the boundary $\partial\mathcal{P}$ of the Newton polygon determines (under topological orbital equivalence) the phase portrait near a singularity of a constrained system in the case where the origin is an equilibrium point of $X$ and a regular point of $\Delta$.

\section{Preliminaries}\label{sec-preliminaries}
\noindent

The results presented in this paper strongly depends on the resolution of singularities. First of all we introduce weighted (or quasi-homogeneous) blow-ups and later we discuss some aspects concerning resolution of singularities of planar constrained systems. For an introduction on blow-ups for planar vector fields we refer \cite{AlvarezFerragutJarque, DumortierLlibreArtes}.

\subsection{Weighted blow-ups}
\noindent

Let $0$ be an isolated equilibrium point of the planar vector field $X$. Given positive integers $\omega_{1}$ and $\omega_{2}$, a \textbf{weighted} (or quasi homogeneous) \textbf{polar blow-up} is the map
\begin{table}[h!]
\centering
\begin{tabular}{lccc}
$\phi_{(\omega_{1},\omega_{2})}:$ & $\mathbb{S}^{1}\times\mathbb{R}$ & $\rightarrow$ & $\mathbb{R}^{2}$, \\
& $(\theta,r)$ & $\mapsto$ & $(r^{\omega_{1}}\cos\theta,r^{\omega_{2}}\sin\theta)$. \\
\end{tabular}
\end{table}

Let $\widehat{X}$ be the vector field defined in $\mathbb{S}^{1}\times\mathbb{R}$ that satisfies $\phi_{(\omega_{1},\omega_{2})*}(\widehat{X}) = X$. Define the vector field $\overline{X}$ as $\overline{X} = \frac{1}{r^{k}}\widehat{X}$, where $k\geq 1$ is the biggest positive integer such that $\overline{X}$ is an analytic vector field. Observe that $\overline{X}$ is defined in a new ambient space and it is invariant in $\mathbb{S}^{1}\times\{0\}$. The set $\mathbb{S}^{1}\times\{0\}$ is called \textbf{exceptional divisor} and the origin $0$ is called \textbf{blow up center}.

In general, during the resolution of singularities it is necessary to iterate a finite number of blowing ups, and considering polar blow ups can lead us to complicated trigonometric expressions. For our purposes, it is way more practical to consider \textbf{weighted} (or quasi homogeneous) \textbf{directional blow-ups}. We define, respectively, the positive $x$-directional and positive $y$-directional blow ups as the maps
\begin{table}[h!]
\centering
\begin{tabular}{lccc}
$\phi_{x,(\omega_{1},\omega_{2})}^{+}:$ & $\mathbb{R}_{\geq0}\times\mathbb{R}$ & $\rightarrow$ & $\mathbb{R}^{2}$, \\
& $(\tilde{x},\tilde{y})$ & $\mapsto$ & $(\tilde{x}^{\omega_{1}},\tilde{x}^{\omega_{2}}\tilde{y})$; \\
\end{tabular}
\end{table}
\begin{table}[h!]
\centering
\begin{tabular}{lccc}
$\phi_{y,(\omega_{1},\omega_{2})}^{+}:$ & $\mathbb{R}\times\mathbb{R}_{\geq0}$ & $\rightarrow$ & $\mathbb{R}^{2}$, \\
& $(\tilde{x},\tilde{y})$ & $\mapsto$ & $(\tilde{y}^{\omega_{1}}\tilde{x},\tilde{y}^{\omega_{2}})$. \\
\end{tabular}
\end{table}

Throughout this paper we only consider the computations in the positive directions because such calculations are analogous when we consider the negative directions. The exceptional divisor will be denoted by $\mathbb{E}_{x} = \{x = 0\}$ and $\mathbb{E}_{y} = \{y = 0\}$ in the $x$ and $y$ directions, respectively. Observe that the polar and directional blow ups are equivalent (see Figure \ref{fig-equivalence-polar-directional}).

\begin{figure}[h!]
	\begin{flushright}
		\begin{center}
			\begin{tikzpicture}
			\node (A) {$(-\displaystyle\frac{\pi}{2},\displaystyle\frac{\pi}{2})\times\mathbb{R}$};
			\node (B) [right of=A] {$\mathbb{R}^{2}$};
			\node (C) [below of=A] {$\mathbb{R}_{\geq0}\times\mathbb{R}$};
			\node (M) [right of=B] {$(0,\pi)\times\mathbb{R}$};
			\node (N) [right of=M] {$\mathbb{R}^{2}$};
			\node (O) [below of=M] {$\mathbb{R}\times\mathbb{R}_{\geq0}$};
			\large\draw[->] (A) to node {\mbox{{\footnotesize $\phi_{(\omega_{1},\omega_{2})}$}}} (B);
			%\draw[->] (X) to node [swap] {\mbox{{\footnotesize $q$}}} (Z);
			\large\draw[->] (C) to node [swap] {\mbox{{\footnotesize $\phi_{x,(\omega_{1},\omega_{2})}^{+}$}}} (B);
			\large\draw[->] (A) to node {\mbox{{\footnotesize $\xi_{x}$}}} (C);
			\large\draw[->] (M) to node {\mbox{{\footnotesize $\phi_{(\omega_{1},\omega_{2})}$}}} (N);
			%\draw[->] (X) to node [swap] {\mbox{{\footnotesize $q$}}} (Z);
			\large\draw[->] (O) to node [swap] {\mbox{{\footnotesize $\phi_{y,(\omega_{1},\omega_{2})}^{+}$}}} (N);
			\large\draw[->] (M) to node {\mbox{{\footnotesize $\xi_{y}$}}} (O);
			\end{tikzpicture}
		\end{center}
	\end{flushright}
	\caption{Equivalence between polar and positive $x$ directional blow up (left), and polar and positive $y$ directional blow up (right).}
\label{fig-equivalence-polar-directional}
\end{figure}
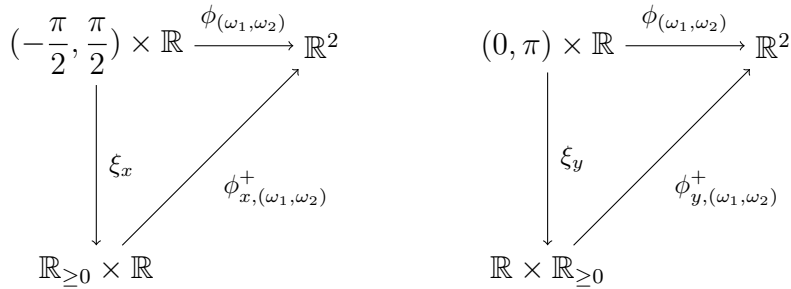

\subsection{Resolution of singularities for planar constrained systems}
\noindent

We emphasize that only isolated singularities of constrained systems are considered. Although a tangency point is a singularity, in this work we focus in the case where $p$ is an equilibrium point of $X$ and/or a singular point of the impasse set. Nevertheless, the theorem of resolution of singularities for 2-dimensional constrained differential systems (see Theorem \ref{teo-resolution-singularities}) also deals with tangency points.

In order to state the theorem of resolution of singularities, one needs to work in a more general category of analytic manifolds with corners. Roughly speaking, a 2-dimensional manifold with corners is a topological space locally modeled by open sets of $(\mathbb{R}_{\geq 0})^{2}$. Just as we mentioned before, after a blow-up in a singularity we obtain a new ambient space, which is a manifold with corners.

A constrained differential system defined in a 2-dimensional manifold with corners is a triple $(\mathcal{M}, \mathcal{X}, \mathcal{I})$, where $\mathcal{M}$ is a 2-dimensional real analytic manifold with corners, $\mathcal{X} = \{(U_{i}, X_{i})\}_{i\in I}$ is an 1-dimensional analytic oriented foliation defined on $\mathcal{M}$ and $\mathcal{I} = \{(U_{i}, \delta_{i})\}_{i\in I}$ is the principal ideal sheaf. On each open set $U_{i}$ of the open covering of $\mathcal{M}$, we associate the diagonalized constrained differential system
$$\delta_{i}(\mathbf{x})\dot{\mathbf{x}} = X_{i}(\mathbf{x}).$$

\begin{theorem}\label{teo-resolution-singularities}
Let $(\mathcal{M}, \mathcal{X}, \mathcal{I})$ be a 2-dimensional real analytic constrained differential system defined on a compact manifold with corners. Then there is a finite sequence of weighted blow ups
$$(\widetilde{\mathcal{M}}, \widetilde{\mathcal{X}}, \widetilde{\mathcal{I}}) = (\mathcal{M}_{n}, \mathcal{X}_{n}, \mathcal{I}_{n}) \xrightarrow{\Phi_{n}} \cdots \xrightarrow{\Phi_{0}} (\mathcal{M}_{0}, \mathcal{X}_{0}, \mathcal{I}_{0}) = (\mathcal{M}, \mathcal{X}, \mathcal{I})$$
such that $(\widetilde{M}, \widetilde{X}, \widetilde{\mathcal{I}})$ is elementary.
\end{theorem}

A proof for Theorem \ref{teo-resolution-singularities} can be given by firstly applying the classical Bendixson--Seidenberg Theorem \cite{Bendixson,Seidenberg} to reduce the singularities of the 1-dimensional foliation $\mathcal{X}$ and later applying results from algebraic geometry on resolution of singularities subordinated to foliations. Indeed, once the equilibrium points of the foliation are elementary (that is, the foliation is Log-Canonical), by \cite{Belotto} there is a resolution for the ideal sheaf $\mathcal{I}$ that preserves the Log-Canonicity of the foliation. For a proof that does not require general results from algebraic geometry and uses weighted blow ups, we refer to \cite{PerezSilva}.

The Theorem \ref{teo-resolution-singularities} is a global theorem, but the study carried throughout this paper is local. Without loss of generality, we always suppose that the constrained system is written in the diagonalized form \eqref{eq-def-constrained-2}. A constrained system defined near a point $p$ will be denoted as a triple $(X,\delta,p)$, where $X$ is the adjoint vector field and $\delta$ is an irreducible analytic function such that $\Delta = \{\delta = 0\}$. \\

\begin{example}\label{exe-resolution}
Consider the following constrained differential equation
\begin{equation}\label{exe-cusp-fold}
xy\dot{x} = y, \ xy\dot{y} = x^{2}.
\end{equation}

The adjoint vector field is given by $X(x,y) = (y,x^{2})$, and therefore there is a cusp singularity at the origin. Moreover, the impasse set is given by $\Delta = \{\delta = xy = 0\}$. We consider the directional weighted blow-ups in the positive $x$ and $y$ directions
$$(x,y) = (\tilde{x}^{2},\tilde{x}^{3}\tilde{y}), \ (x,y) = (\bar{y}^{2}\bar{x},\bar{y}^{3}),$$
respectively. Thus we obtain the constrained system
$$\tilde{y}\dot{\tilde{x}} = \displaystyle\frac{\tilde{x}\tilde{y}}{2}, \ \tilde{y}\dot{\tilde{y}} = 1 - \displaystyle\frac{3\tilde{y}^{2}}{2}$$
in the positive $x$ direction and
$$\bar{x}\dot{\bar{x}} = 1 - \displaystyle\frac{2}{3}\bar{x}^{3}, \ \bar{x}\dot{\bar{y}} = \displaystyle\frac{\bar{x}^{2}\bar{y}}{3}$$
in the positive $y$ direction. In Section \label{sec-arnold} we briefly discuss how to make a good choice of the weight vector. See Figure \ref{fig-exe-cusp-dobra}.
\end{example}

\begin{figure}[h]
  \center{\includegraphics[width=0.65\textwidth]{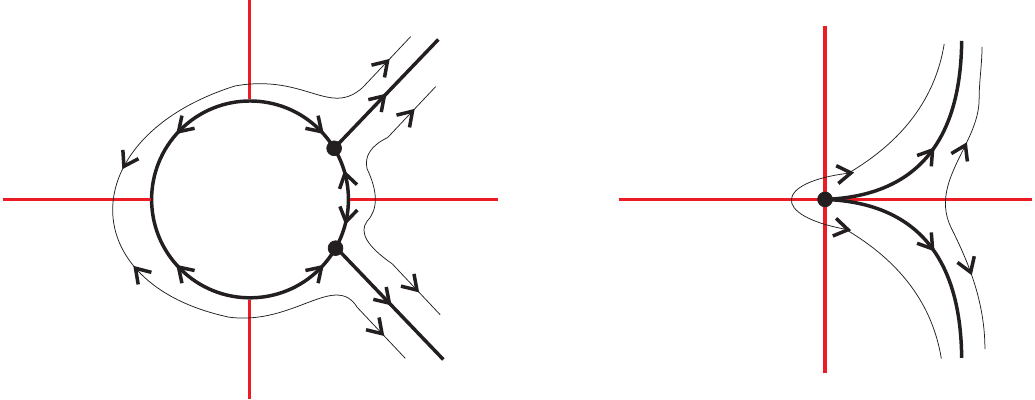}}\\
  \caption{Constrained system \eqref{exe-cusp-fold} after (left) and before (right) the weighted blow-up at the origin.}\label{fig-exe-cusp-dobra}
\end{figure}

\section{The construction of the orbital $C^{0}$-equivalence near the exceptional divisor}\label{sec-construction}
\noindent

Let $(X,\delta,p_{1})$ be a constrained differential system defined near an impasse point $p_{1}\in\Delta_{\delta}$. Suppose that $p_{1}$ is an equilibrium point of $X$ or a singular point of the impasse set $\Delta_{\delta}$. In this point is applied a suitable finite sequence of weighted blowing-ups, whose composition is simply denote by $\Phi$. According to Theorem \ref{teo-resolution-singularities}, at the end of the process, we obtain a constrained system $(\widetilde{X}, \widetilde{\delta}, D_{1})$ such that all the points in a neighborhood of the exceptional divisor $D_{1} = \Phi^{-1}(p_{1})$ are elementary. Moreover, $D_{1}$ is an invariant set of $\widetilde{X}$.

Analogously, let $(Y,\gamma,p_{2})$ be a planar constrained system where $p_{2}\in\Delta_{\gamma}$ is an equilibrium point of the adjoint vector field $Y$ or a singular point of the impasse set $\Delta_{\gamma}$. After a finite number of weighted blowing-ups (whose composition we denote by $\Psi$), we obtain the strict transformed $(\widetilde{Y}, \widetilde{\gamma}, D_{2})$ where all the points near the exceptional divisor $D_{2} = \Psi^{-1}(p_{2})$ are elementary and $D_{2}$ is an invariant set of $\widetilde{Y}$.

The main goal is to use the resolution of singularities in the topological classification of phase portraits. In order to achieve this objective, firstly we give conditions based in the resolution of singularities to assure the equivalence between two constrained systems, that is, we study conditions to assure the existence of an orientation preserving orbital $C^{0}$-equivalence $H$ between $(X, \delta, p_{1})$ and $(Y, \gamma, p_{2})$. The Proposition \ref{prop-c0-eq} will say that if there is an orientation preserving orbital $C^{0}$-equivalence $\widetilde{H}$ between $(\widetilde{X}, \widetilde{\delta}, D_{1})$ and $(\widetilde{Y}, \widetilde{\gamma}, D_{2})$ in a neighborhood of the exceptional divisor, then $(X, \delta, p_{1})$ and $(Y, \gamma, p_{2})$ are equivalents. See Figure \ref{fig-diag-prop}.

\begin{figure}[h!]
\begin{flushright}
\begin{center}
\begin{tikzpicture}
\node (A) {$(\widetilde{X}, \widetilde{\delta}, D_{1})$};
\node (B) [right of=A] {$(\widetilde{Y}, \widetilde{\gamma}, D_{2})$};
\node (C) [below of=A] {$(X, \delta, p_{1})$};
\node (D) [below of=B] {$(Y, \gamma, p_{2})$};
\large\draw[->] (A) to node {\mbox{{\footnotesize $\Phi$}}} (C);
%\draw[->] (X) to node [swap] {\mbox{{\footnotesize $q$}}} (Z);
\large\draw[->] (B) to node {\mbox{{\footnotesize $\Psi$}}} (D);
\large\draw[->] (A) to node {\mbox{{\footnotesize $\widetilde{H}$}}} (B);
\large\draw[->] (C) to node {\mbox{{\footnotesize $H$}}} (D);
\end{tikzpicture}
\end{center}
\end{flushright}
\caption{Diagram of Proposition \ref{prop-c0-eq}.}\label{fig-diag-prop}
\end{figure}

Let $W_{1}\ni p_{1}$ and $W_{2}\ni p_{2}$ be neighborhoods of the blow-up centers $p_{1}$ and $p_{2}$, respectively. Observe that $\widetilde{W}_{1} = \Phi^{-1}(W_{1}\backslash\{p_{1}\})$ and $\widetilde{W}_{2} = \Psi^{-1}(W_{2}\backslash\{p_{2}\})$ are open sets. Define the open sets $\widetilde{U} = \widetilde{W}_{1} \cup D_{1}$ and $\widetilde{V} = \widetilde{W}_{2} \cup D_{2}$ and suppose that there is a homeomorphism (possibly taking smaller neighborhoods $W_{1}$ and $W_{2}$ if it is necessary) $\widetilde{H}:\widetilde{U}\rightarrow \widetilde{V}$ satisfying the following conditions:

\begin{description}
  \item[(H1)] $\widetilde{H}(D_{1}) = D_{2}$;
  \item[(H2)] $\widetilde{H}(\widetilde{\Delta}_{\delta}) = \widetilde{\Delta}_{\gamma}$;
  \item[(H3)] Let $p\in\widetilde{U}$ and denote by $\phi_{\widetilde{X}}\big{(}p,t\big{)}$ the flow of $\widetilde{X}$ through $p$. Suppose that $\phi_{\widetilde{X}}\big{(}p,[0,s]\big{)} \subset \widetilde{U}\backslash\widetilde{\Delta}_{\delta}$, for some $s >0$. Then there is $\hat{s} > 0$ such that
      $$\widetilde{H}\Bigg{(}\phi_{\widetilde{X}}\big{(}p,[0,s]\big{)}\Bigg{)} = \phi_{\widetilde{Y}}\Bigg{(}\widetilde{H}(p),[0,\hat{s}]\Bigg{)}.$$
\end{description}

%\begin{figure}[h!]
%  \center{\includegraphics[width=0.60\textwidth]{fig-hypotheses-h1-h2-h3}}\\
%  \caption{Homeomorphism $\widetilde{H}$ that satisfies the hypotheses (H1), (H2) and (H3). The impasse set is highlighted in red.}\label{fig-hypotheses-h1-h2-h3}
%\end{figure}

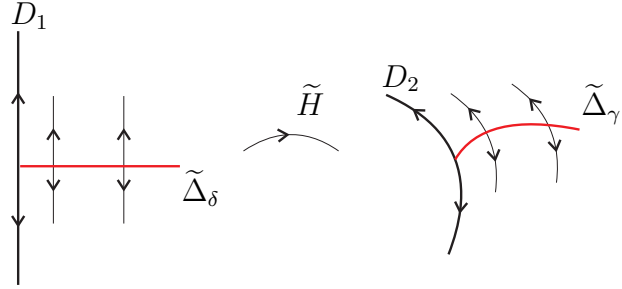
\begin{figure}[h]\center{
\begin{overpic}[width=0.60\textwidth]{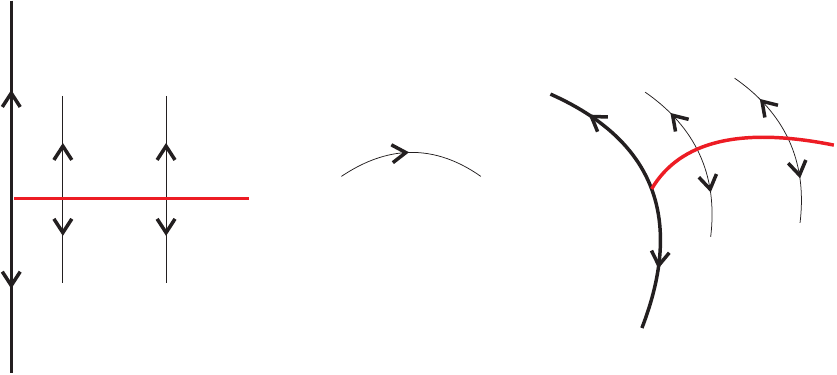}
%\begin{overpic}[grid,tics=10,width=0.60\textwidth]{fig-hypotheses-h1-h2-h3}
\put(50,30){$\widetilde{H}$}
\put(0,46){$D_{1}$}
\put(65,35){$D_{2}$}
\put(30,15){$\widetilde{\Delta}_{\delta}$}
\put(100,30){$\widetilde{\Delta}_{\gamma}$}
\end{overpic}}
\caption{Homeomorphism $\widetilde{H}$ that satisfies the hypotheses (H1), (H2) and (H3). The impasse set is highlighted in red.}
\label{fig-hypotheses-h1-h2-h3}
\end{figure}

\begin{proposition}\label{prop-c0-eq}
If there is a homeomorphism $\widetilde{H}$ satisfying the conditions (H1), (H2) and (H3) above, then there are neighborhoods $U\ni p_{1}$ and $V\ni p_{2}$ and an orientation preserving orbital $C^{0}$-equivalence $H:U\rightarrow V$ between the constrained systems $(X, \delta, p_{1})$ and $(Y, \gamma, p_{2})$.
\end{proposition}
\begin{proof}
Denote $U = \Phi(\widetilde{U})$ and $V = \Psi(\widetilde{V})$ and define the map $H:U\rightarrow V$ as
$$H(p_{1}) = p_{2}; \ H(p) = \Psi\circ \widetilde{H} \circ \Phi^{-1}(p), \ p\in \Phi(\widetilde{W}_{1}).$$

We claim that $H$ is an orientation preserving orbital $C^{0}$-equivalence between $(X, \delta, p_{1})$ and $(Y, \gamma, p_{2})$. Recall that $\Phi$ restricted to $\widetilde{W}_{1} = \Phi^{-1}(W_{1}\backslash\{p_{1}\})$ is an analytic diffeomorphism that maps $\widetilde{\Delta}_{\delta}$ into $\Delta_{\delta}$, and it maps phase curves of $\widetilde{X}$ into phase curves of $X$, preserving orientation. Analogous properties for $\Psi$ restricted to $\widetilde{W}_{2} = \Psi^{-1}(W_{2}\backslash\{p_{2}\})$ are true. Moreover, $\Phi(D_{1}) = p_{1}$, $\widetilde{H}(D_{1}) = D_{2}$ and $\Psi(D_{2}) = p_{2}$.

A straightforward computation shows that $H$ maps impasse set into impasse set. Furthermore, if $\phi_{X}\big{(}p,t\big{)}$ is the flow of $X$ through $p\in U$ and $\phi_{X}\big{(}p,[0,s]\big{)} \subset U\backslash\Delta_{\delta}$ for some $s >0$, then there is $\hat{s} > 0$ such that
$$H\Bigg{(}\phi_{X}\big{(}p,[0,s]\big{)}\Bigg{)} = \phi_{Y}\Bigg{(}H(p),[0,\hat{s}]\Bigg{)},$$
and therefore $H$ is an orientation preserving orbital $C^{0}$-equivalence.
\end{proof}

Since the existence of $\widetilde{H}$ assures the existence of an orientation preserving orbital $C^{0}$-equivalence $H$ between $(X, \delta, p_{1})$ and $(Y, \gamma, p_{2})$, now we will provide a condition based on the resolution of singularities that guarantees the existence of such homeomorphism $\widetilde{H}$. This condition is the so called \emph{elementary singularity scheme}. The next theorem says that if $(\widetilde{X}, \widetilde{\delta}, D_{1})$ and $(\widetilde{Y}, \widetilde{\gamma}, D_{2})$ have the same elementary singularity scheme, then such triples are (orientation preserving) orbitally $C^{0}$-equivalent near the exceptional divisor. Such result and the Proposition \ref{prop-c0-eq} will assure that $(X, \delta, p_{1})$ and $(Y, \gamma, p_{2})$ are orbitally $C^{0}$-equivalent near the blow-up center.

\subsection{Basic points, basic singular interval and elementary singularity scheme}
\noindent

Let $(X,\delta,p_{1})$ be a planar constrained system where $p_{1}\in\Delta_{\delta}$ is an equilibrium point of $X$ or a singular point of the impasse set. Consider its strict transformed $(\widetilde{X}, \widetilde{\delta}, D_{1})$ where all points in the exceptional divisor $D_{1}$ are elementary.

We say that a coordinate system is \textbf{adapted to the exceptional divisor} if in such coordinate system the exceptional divisor is given by $\mathbb{E}_{x} = \{x = 0\}$, $\mathbb{E}_{y} = \{y = 0\}$ or $\mathbb{E}_{xy} = \{xy = 0\}$ (see Definition 5.1, \cite{Dumortier}).

Define the following vector fields with restricted domain, where $\alpha,\beta\in\{-1,1\}$:
\begin{enumerate}
  \item $V_{1}^{\alpha,\beta}$: $\alpha x\frac{\partial}{\partial x} + \beta y\frac{\partial}{\partial y},$ \ $y\geq 0$. See Figure \ref{fig-bvf-1};
  \item $V_{2}^{\alpha,\beta}$: $\beta x^{2}\frac{\partial}{\partial x} + \alpha y\frac{\partial}{\partial y},$ \ $y\geq 0$. See Figure \ref{fig-bvf-2};
  \item $V_{3}^{\alpha}$: $\alpha x\frac{\partial}{\partial x} - \alpha y\frac{\partial}{\partial y},$ \ $x\geq 0$ and $y\geq 0$. See Figure \ref{fig-bvf-3-4};
  \item $V_{4}^{\alpha}$: $\alpha y\frac{\partial}{\partial y},$ \ $0\leq x \leq 1$ and $y\geq 0$. See Figure \ref{fig-bvf-3-4}.
\end{enumerate}

Analogously, define the following constrained differential systems with restricted domain, where $\alpha,\beta\in\{-1,1\}$:
\begin{enumerate}
  \item $C_{1}^{\alpha,\beta}$: $x\dot{x} = \alpha x, x\dot{y} = \beta y,$ \ $y\geq 0$. See Figure \ref{fig-bcs-1};
  \item $C_{2}^{\alpha,\beta}$: $x\dot{x} = \beta x^{2}, x\dot{y} = \alpha y,$ \ $y\geq 0$. See Figure \ref{fig-bcs-2};
  \item $C_{3}^{\alpha}$: $x\dot{x} = 0,  x\dot{y} = \alpha,$ \ $0\leq x \leq 1$ and $y\geq 0$. See Figure \ref{fig-bcs-3};
  \item $C_{4}^{\alpha}$: $x\dot{x} = 0, x\dot{y} = \alpha y,$ \ $0\leq x \leq 1$ and $y\geq 0$. See Figure \ref{fig-bcs-3}.
\end{enumerate}

\begin{figure}[h!]
  \center{\includegraphics[width=0.75\textwidth]{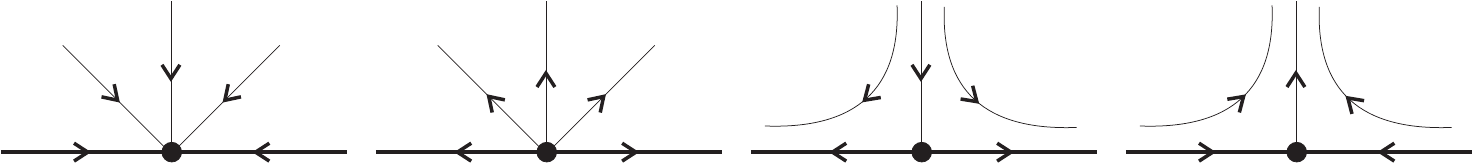}}\\
  \caption{Basic singularities of the form $V^{\alpha,\beta}_{1}$.}\label{fig-bvf-1}
\end{figure}

\begin{figure}[h!]
  \center{\includegraphics[width=0.75\textwidth]{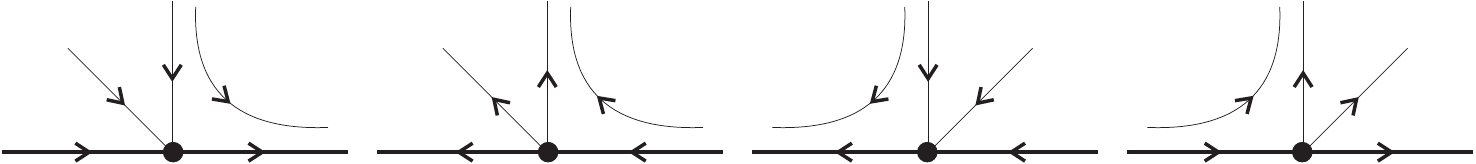}}\\
  \caption{Basic singularities of the form $V^{\alpha,\beta}_{2}$.}\label{fig-bvf-2}
\end{figure}

\begin{figure}[h!]
  \center{\includegraphics[width=0.27\textwidth]{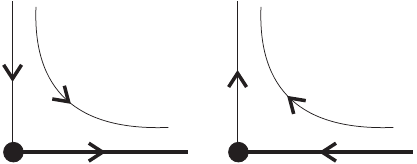}\hspace{0.5cm}\includegraphics[width=0.35\textwidth]{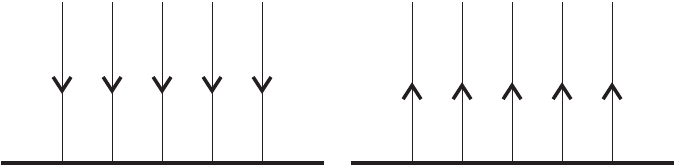}}\\
  \caption{Basic singularities of the form $V^{\alpha}_{3}$ (left) and $V^{\alpha}_{4}$ (right).}\label{fig-bvf-3-4}
\end{figure}

\begin{figure}[h!]
  \center{\includegraphics[width=0.75\textwidth]{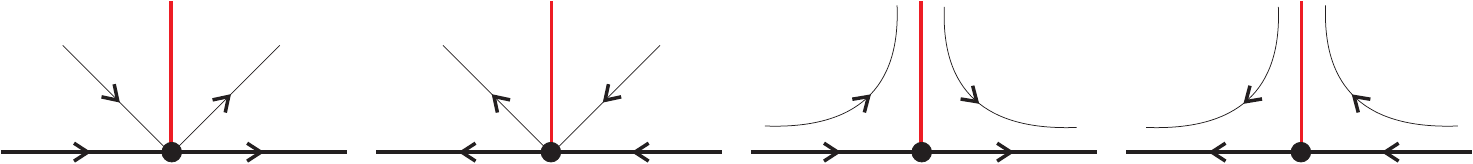}}\\
  \caption{Basic singularities of the form $C^{\alpha,\beta}_{1}$. The impasse set is highlighted in red.}\label{fig-bcs-1}
\end{figure}

\begin{figure}[h!]
  \center{\includegraphics[width=0.75\textwidth]{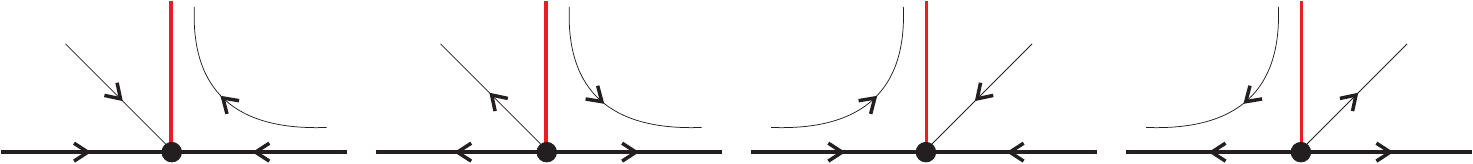}}\\
  \caption{Basic singularities of the form $C^{\alpha,\beta}_{2}$. The impasse set is highlighted in red.}\label{fig-bcs-2}
\end{figure}

\begin{figure}[h!]
  \center{\includegraphics[width=0.40\textwidth]{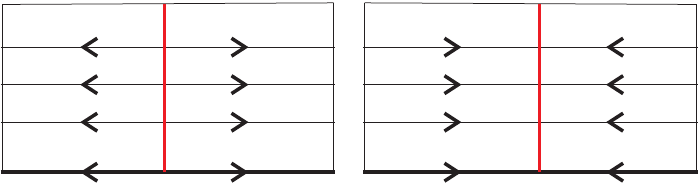}\hspace{0.5cm}\includegraphics[width=0.40\textwidth]{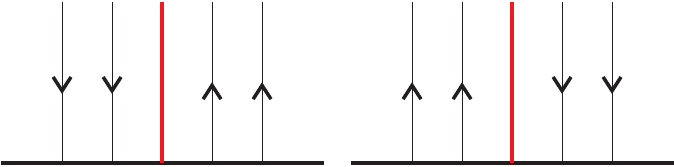}}\\
  \caption{Basic singularities of the form $C^{\alpha}_{3}$ (left) $C^{\alpha}_{4}$ (right). The impasse set is highlighted in red.}\label{fig-bcs-3}
\end{figure}

%\begin{figure}[h]
%  \center{\includegraphics[width=0.75\textwidth]{fig-basic-points-1}}\\
%  \caption{From the left to the right: $BP_{1}$, $BP_{2}$, $BP_{3}$ and $BP_{4}$.}\label{fig-basic-points-1}
%\end{figure}

%\begin{figure}[h]
%  \center{\includegraphics[width=0.50\textwidth]{fig-basic-points-2}}\\
%  \caption{From the left to the right: $BP_{5}$, $BP_{6}$ and $BI$.}\label{fig-basic-points-2}
%\end{figure}

\begin{definition}
A point $p\in D_{1}$ is a \textbf{basic singularity} if in a neighborhood $U$ of $p$ there is a coordinate system adapted to the divisor such that the constrained system is orientation preserving orbitally $C^{0}$-equivalent to $V_{i}^{\alpha,\beta}$ or $C_{j}^{\alpha,\beta}$, where $i,j = 1,2,3$.
\end{definition}

\begin{definition}
A \textbf{basic singular interval} is an arc $d\subset D_{1}$ such that all the points in $d$ are equilibrium points and in a neighborhood $U$ of $d$ there is a coordinate system adapted to the divisor such that the constrained system is orientation preserving orbitally $C^{0}$-equivalent to $V_{4}^{\alpha}$ or $C_{4}^{\alpha}$.
\end{definition}

Since the excepcional divisor is a finite union of smooth arcs, one can enumerate them (without loss of generality) in the clockwise sense. We suppose that $d_{i+1}$ follows $d_{i}$ and $d_{n + 1} = d_{1}$. With this orientation, one can also order the basic singularities and the basic singular intervals in the excepcional divisor. The arrangement of basic singularities and basic singular intervals in the divisor defines a finite word constructed with the alphabet
$$\{V_{i}^{\alpha,\beta}\}_{i = 1}^{2}\cup\{C_{j}^{\alpha,\beta}\}_{j = 1}^{2}\cup\{V_{i}^{\alpha}\}_{i = 3}^{4}\cup\{C_{j}^{\alpha}\}_{j = 3}^{4}.$$

Such arrangement defines an equivalence relation in the set $\Sigma$ of triples $(\widetilde{X}, \widetilde{\delta}, D_{1})$ in the following way: two triples $(\widetilde{X}, \widetilde{\delta}, D_{1})$ and $(\widetilde{Y}, \widetilde{\gamma}, D_{2})$ are equivalent if, and only if, the word associated to $(\widetilde{X}, \widetilde{\delta}, D_{1})$ can be changed to the word associated to $(\widetilde{Y}, \widetilde{\gamma}, D_{2})$ by cyclic permutations.

\begin{definition}
An equivalence class in $\Sigma$ is called \textbf{elementary singularity scheme}.
\end{definition}

\begin{example}
The elementary singularity scheme of the constrained system Example \ref{exe-resolution} is given by the word $C^{1}_{3}V^{-1,1}_{1}C^{1}_{3}V^{-1,1}_{1}C^{1}_{3}C^{-1}_{3}$.
\end{example}

The next theorem is true under the hypothesis that the elementary singularity schemes do not contain only words written with symbols of the form $V_{3}^{\alpha}$ or $C_{3}^{\alpha}$. This is equivalent to require the following. Let $(\widetilde{X}, \widetilde{\delta}, D_{1})$ and $(\widetilde{Y}, \widetilde{\gamma}, D_{2})$ be the strict transformed of the constrained systems $(X, \delta, p_{1})$ and $(Y, \gamma, p_{2})$, respectively, where $p_{1}\in \Delta_{\delta}$ and $p_{2}\in \Delta_{\gamma}$. If $p_{1}$ and $p_{2}$ are equilibrium points of the adjoint vector field $X$ and $Y$, respectively, then both $p_{1}$ and $p_{2}$ have characteristic orbit. In other words, $p_{1}$ and $p_{2}$ are neither a center nor a focus. See Figure \ref{fig-forbidden-cases}.

\begin{definition}
An elementary singularity scheme is \textbf{degenerated} if the word associated to the triple $(\widetilde{X}, \widetilde{\delta}, D_{1})$ only contains symbols of the form $V_{3}^{\alpha}$ or $C_{3}^{\alpha}$. Otherwise, we say that the elementary singularity scheme is \textbf{non-degenerated}.
\end{definition}

\begin{figure}[h]
  \center{\includegraphics[width=0.50\textwidth]{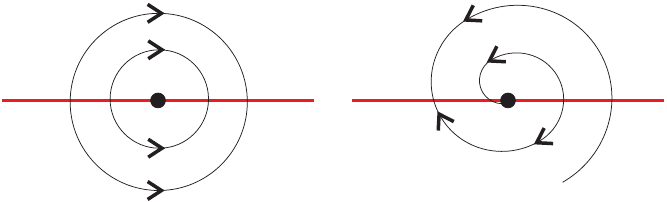}}\\
  \caption{Singularities of a constrained system. The equilibrium point of the adjoint vector field does not have characteristic orbit.}\label{fig-forbidden-cases}
\end{figure}

\subsection{Proof of Theorem \ref{coroo-eq-conjugacy}: The construction of the orbital $C^{0}$-equivalence}
\noindent

The next lemma is a well-known result from general topology and it will be useful in the proof of Proposition \ref{teo-eq-conjugacy}.

\begin{lemma}\label{lemma-pasting-lemma} (The Pasting Lemma, Theorem 18.3, \cite{Munkres}) Let $M, N$ be topological spaces and let $A_{1}, A_{2}\subset M$ be open (or closed) subspaces. Denote $A = A_{1}\cap A_{2}$. Let $f: A_{1}\rightarrow f(A_{1})\subset N$ and $g:A_{2}\rightarrow g(A_{2})\subset N$ be homeomorphisms such that $f|_{A} = g|_{A}$. Then there is a homeomorphism $H:A_{1}\cup A_{2}\rightarrow H(A_{1}\cup A_{2})\subset N$ such that $H|_{A_{1}} = f$ and $H|_{A_{2}} = g$.
\end{lemma}

\begin{proposition}\label{teo-eq-conjugacy} Suppose that the triples $(\widetilde{X}, \widetilde{\delta}, D_{1})$ and $(\widetilde{Y}, \widetilde{\gamma}, D_{2})$ have the same non-degenerated elementary singularity scheme. Then there is an orientation preserving orbital $C^{0}$-equivalence $H$ between $(\widetilde{X}, \widetilde{\delta}, D_{1})$ and $(\widetilde{Y}, \widetilde{\gamma}, D_{2})$ such that $H(D_{1}) = D_{2}$.
\end{proposition}
\begin{proof}
Given an arc $c_{i}\subset D_{1}$, we have its respective $\widehat{c}_{i}\subset D_{2}$. The construction of the homeomorphism starts near a basic point $p$ (or basic singular interval) of $(\widetilde{X}, \widetilde{\delta}, D_{1})$ that is not of the form $V_{3}^{\alpha}$ or $C_{3}^{\alpha}$. Hence there is an equivalent point $\widehat{p}$ (or basic singular interval) of $(\widetilde{Y}, \widetilde{\gamma}, D_{2})$. This means that there is adapted coordinates
\begin{center}
$h: U\rightarrow \mathbb{R}^{2}, \ \ \widehat{h}: \widehat{U}\rightarrow \mathbb{R}^{2},$
\end{center}
for $(\widetilde{X}, \widetilde{\delta}, D_{1})$ and $(\widetilde{Y}, \widetilde{\gamma}, D_{2})$, respectively, where $U\ni p$ and $\widehat{U}\ni \widehat{p}$. Denote $h(U) = \widehat{h}(\widehat{U}) = W$. The neighborhoods $U$ and $\widehat{U}$ are chosen in such way that there is only one basic point (or only one basic singular interval).

%\begin{figure}[h!]
%  \center{\includegraphics[width=0.5\textwidth]{fig-dumortier}}\\
%  \caption{Neighborhoods $V$, $V'$ and $U$.}\label{fig-dumortier}
%\end{figure}

\begin{figure}[h]\center{
\begin{overpic}[width=0.50\textwidth]{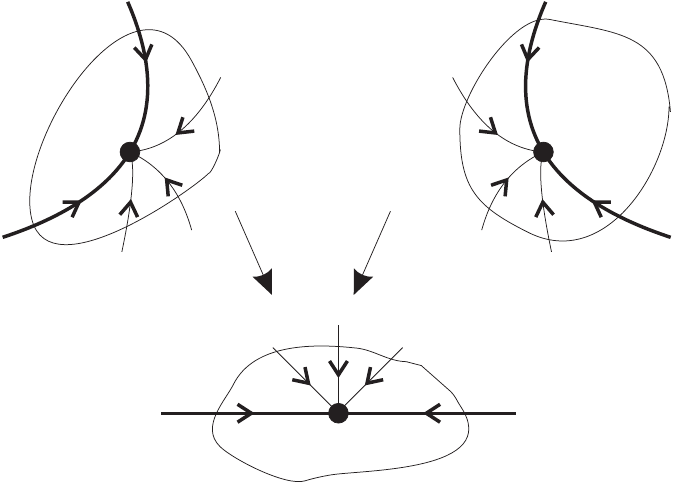}
%\begin{overpic}[grid,tics=10,width=0.50\textwidth]{fig-dumortier}
\put(10,28){$U$}
\put(90,30){$\widehat{U}$}
\put(38,38){$h$}
\put(52,38){$\widehat{h}$}
\put(68,18){$W$}
\end{overpic}}
\caption{Neighborhoods $U$, $\widehat{U}$ and $W$.}
\label{fig-dumortier}
\end{figure}
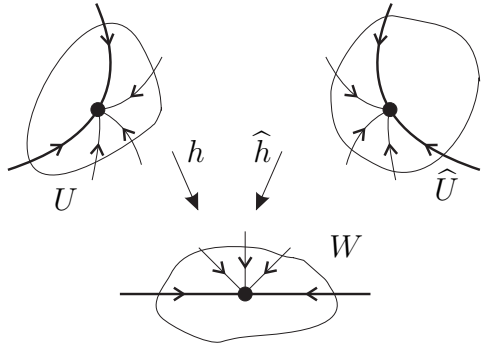

If we start the construction of the homeomorphism from a basic point, we take the open set $W$ in such way that it is transversal to the coordinate axes. On the other hand, if the construction starts from a basic singular interval, the neighborhood $W$ is chosen in such way that it is transversal to the lines $\{x = 0\}$ and $\{x = 1\}$.

Therefore, we have that $h^{-1}(\partial W)$ is transversal to $D_{1}\cap U$. This is also true when we take $(\widehat{h})^{-1}(\partial W)$ and $D_{2}\cap \widehat{U}$. Applying this reasoning to all basic points and basic singular intervals in the exceptional divisor, a sufficiently small neighborhood of the exceptional divisor $D_{1}$ is divided into a finite number of sectors. If this sector is not locally given by $V_{4}^{\alpha}$, we then subdivide such sector by the segment $h^{-1}(\{x = 0\})$. We apply the same reasoning for $(\widetilde{Y}, \widetilde{\gamma}, D_{2})$.

These sectors are bounded by arcs of the excepcional divisor, phase curves, impasse curves or curves transversal to the divisor. Considering $\alpha\in\{-1,1\}$, these sectors are written in local coordinates as follows:

\begin{itemize}
  \item Sector $S_{1}^{\alpha}$: $\alpha(x^{2} + y^{2})\frac{\partial}{\partial x}$, where $-1\leq x \leq 1$ and $y\geq 0$. This sector is equivalent to the basic point $V^{\alpha}_{3}$;
  \item Sector $S_{2}^{\alpha}$: $\alpha x\frac{\partial}{\partial x} + \alpha y\frac{\partial}{\partial y}$, where $-1\leq x \leq 1$ and $y\geq 0$;
  \item Sector $S_{3}^{\alpha}$: $\alpha y\frac{\partial}{\partial y}$, where $-1\leq x \leq 1$ and $y\geq 0$;
  \item Sector $F^{\alpha}$: $\alpha\frac{\partial}{\partial x}$, where $-1\leq x \leq 1$ and $y\geq 0$;
  \item Sector $SI_{1}^{\alpha}$: $x\dot{x} = \alpha x, x\dot{y} = \alpha y$, where $\alpha x \geq 0$ and $y \geq 0$;
  \item Sector $SI_{2}^{\alpha}$: $x\dot{x} = -\alpha x, x\dot{y} = -\alpha y$, where $\alpha x \geq 0$ and $y \geq 0$;
  \item Sector $SI_{3}^{\alpha}$: $x\dot{x} = \alpha x, x\dot{y} = -\alpha y$, where $\alpha x \geq 0$ and $y \geq 0$;
  \item Sector $SI_{4}^{\alpha}$: $x\dot{x} = -\alpha x, x\dot{y} = \alpha y$, where $\alpha x \geq 0$ and $y \geq 0$;
  \item Sector $SI_{5}^{\alpha}$: $x\dot{x} = 0, x\dot{y} = \alpha y$, where $\alpha x \geq 0$ and $y \geq 0$;
  \item Sector $SI_{6}^{\alpha}$: $x\dot{x} = 0, x\dot{y} = -\alpha y$, where $\alpha x \geq 0$ and $y \geq 0$;
  \item Sector $FI^{\alpha}$: $x\dot{x} = \alpha, x\dot{y} = 0$, where $-1\leq x \leq 1$ and $y\geq 0$.
\end{itemize}

We remark that the sector $S_{1}^{\alpha}$ is $C^{0}$-equivalent to $V_{3}^{\alpha}$.

\begin{figure}[h!]
  \center{\includegraphics[width=0.75\textwidth]{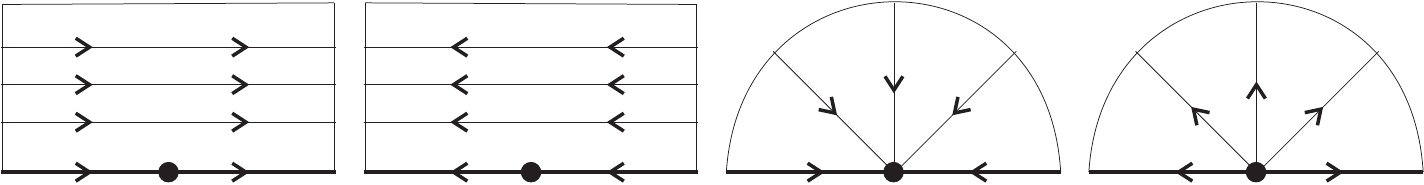}}\\
  \caption{Sectors $S_{1}^{\alpha}$ (left) and $S_{2}^{\alpha}$ (right).}\label{fig-sector-s1-s2}
\end{figure}

\begin{figure}[h!]
  \center{\includegraphics[width=0.75\textwidth]{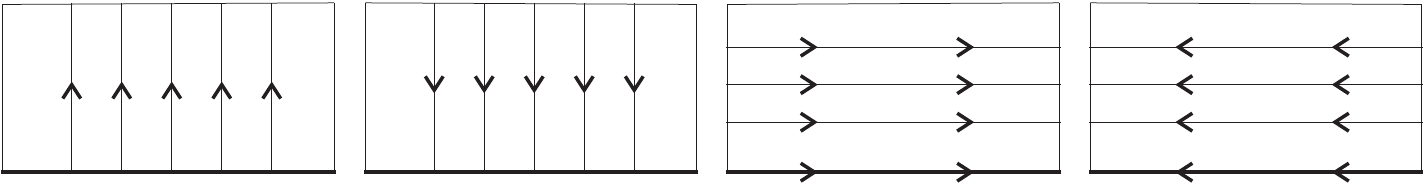}}\\
  \caption{Sectors $S_{3}^{\alpha}$ (left) and $F^{\alpha}$ (right).}\label{fig-sector-s3-fb}
\end{figure}

\begin{figure}[h!]
  \center{\includegraphics[width=0.65\textwidth]{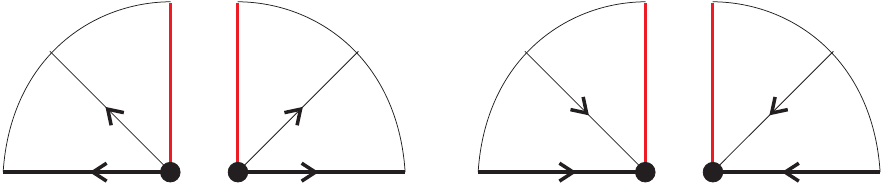}}\\
  \caption{Sectors $SI_{1}^{\alpha}$ (left) and $SI_{2}^{\alpha}$ (right). The impasse set is highlighted in red.}\label{fig-sector-si1-si2}
\end{figure}

\begin{figure}[h!]
  \center{\includegraphics[width=0.65\textwidth]{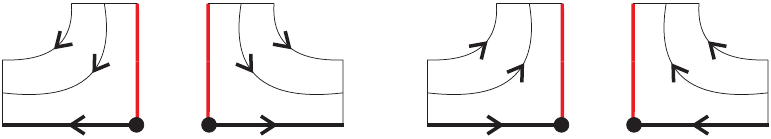}}\\
  \caption{Sectors $SI_{3}^{\alpha}$ (left) and $SI_{4}^{\alpha}$ (right). The impasse set is highlighted in red.}\label{fig-sector-si3-si4}
\end{figure}

\begin{figure}[h!]
  \center{\includegraphics[width=0.65\textwidth]{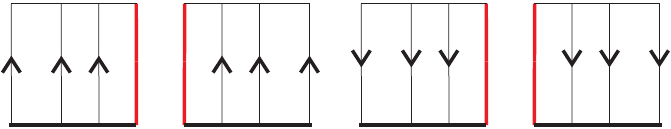}}\\
  \caption{Sectors $SI_{5}^{\alpha}$ (left) and $SI_{6}^{\alpha}$ (right). The impasse set is highlighted in red.}\label{fig-sector-si5-si6}
\end{figure}

\begin{figure}[h!]
  \center{\includegraphics[width=0.50\textwidth]{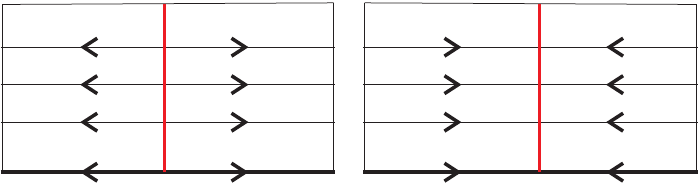}}\\
  \caption{Sector $FI^{\alpha}$. The impasse set is highlighted in red.}\label{fig-sector-fbi}
\end{figure}

Consider once again the point $p$ that we took in the beginning of this proof. Recall that this point is neither $V_{3}^{\alpha}$ nor $C_{3}^{\alpha}$. We know that $(\widehat{h})^{-1}\circ h: U\rightarrow \widehat{U}$ is a homeomorphism such that $h(p) = \widehat{p}$, it maps phase curves of $\widetilde{X}$ into phase curves of $\widetilde{Y}$ and it sends impasse set into impasse set. Moreover, $(\widehat{h})^{-1}\circ h$ sends excepcional divisor into excepcional divisor. Then we have constructed a homeomorphism for this first sector.

The adjacent sector (in the clockwise sense) must be of the form $F^{\alpha}$ or $FI^{\alpha}$. In this second sector it is defined a $C^{0}$-equivalence $\widehat{\widehat{h}}$ between $(\widetilde{X}, \widetilde{\delta}, D_{1})$ and $(\widetilde{Y}, \widetilde{\gamma}, D_{2})$. Observe now that on $\partial U$ the homeomorphisms $(\widehat{h})^{-1}\circ h$ and $\widehat{\widehat{h}}$ coincide.

\begin{figure}[h!]
  \center{\includegraphics[width=0.50\textwidth]{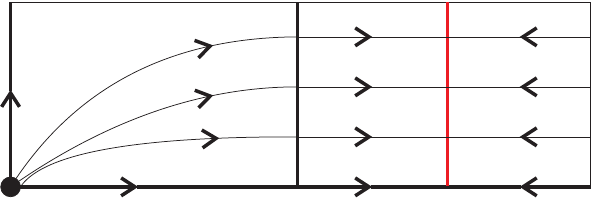}}\\
  \caption{Gluing two adjacent sectors. The impasse set is highlighted in red.}\label{fig-sector-glueing}
\end{figure}

We remark that each phase curve intersects $\partial U$ in only one point. Now, by the Pasting Lemma \ref{lemma-pasting-lemma}, there is an homeomorphism defined in both sectors that maps phase curves of $\widetilde{X}$ into phase curves of $\widetilde{Y}$, maps impasse set into impasse set and sends excepcional divisor into excepcional divisor.

Between each sector of the form $S_{i}^{\alpha}$ or $SI_{j}^{\alpha}$ it may have a sector of the form $F^{\alpha}$ or $FI^{\alpha}$, and we need to glue all these sectors around the excepcional divisor. At the boundary of each sector, the Pasting Lemma \ref{lemma-pasting-lemma} is applied.

At the end of the process, we must glue the last sector and the first sector in such a way that we ``close'' the construction of the homeomorphism, given that we are constructing such a homeomorphism around the exceptional divisor. Observe that, in general, such construction only can be closed if the basic point $p$ (or basic singular interval) that we took in the beginning of the proof is not of the form $V_{3}^{\alpha}$ or $C_{3}^{\alpha}$. Geometrically, it means that we must avoid the center-focus case.

Thus we have constructed an homeomorphism $H$ such that $H(D_{1}) = D_{2}$, $H$ maps phase curves of $\widetilde{X}$ into phase curves of $\widetilde{Y}$ and $H$ sends $\widetilde{\Delta}_{\delta}$ into $\widetilde{\Delta}_{\gamma}$. Then $H$ is the orbital $C^{0}$-equivalence desired.
\end{proof}

Combining the Propositions \ref{prop-c0-eq} and \ref{teo-eq-conjugacy} we obtain the Theorem \ref{coroo-eq-conjugacy}, which gives a condition to assure the existence of an (orientation preserving) orbital $C^{0}$-equivalence between two constrained differential systems, and such condition is based on the resolution of singularities. Since the classification of phase portraits depends on the notion of equivalence adopted, this result plays an important role in what follows.

\section{Constrained systems with singular impasse sets and the proof of Theorem \ref{teo-singular-curves-ade-resolution}}\label{sec-arnold}
\noindent

The main goal of this section is to give a classification of phase portraits of planar constrained systems near singular points of the impasse curve. We recall from the introduction that the set $\mathcal{A}$ is the set of constrained systems $(X, \delta)$ such that:
\begin{enumerate}
  \item The adjoint vector field is constant, that is, it satisfies $X(x,y) = (a,b)$, $a,b\in\mathbb{R}$ and $a^{2} + b^{2} \neq 0$.
  \item The function $\delta$ is one of the curves given by Arnold's ADE-classification (see Table \ref{tab-arnold-ade}).
\end{enumerate}

\begin{table}[ht]
\centering
\begin{tabular}{|l|c|c|}
  \hline
  Type & Normal Form & Codimension \\ \hline
  $A_{k}$, $k\geq 1$ & $x^{2}\pm y^{k+1}$ & $k-1$ \\
  $D_{k}$, $k\geq 4$ & $x^{2}y\pm y^{k-1}$ & $k-1$ \\
  $E_{6}$ & $x^{3} \pm y^{4}$ & $5$ \\
  $E_{7}$ & $y^{3} - yx^{3}$ & $6$ \\
  $E_{8}$ & $x^{3} + y^{5}$ & $7$ \\
  \hline
\end{tabular}
\caption{ADE-type singularities \cite{Arnold}.}\label{tab-arnold-ade}
\end{table}

We chose a coordinate system such that $\delta$ is one of the functions given by Table \ref{tab-arnold-ade}. Since the adjoint vector field and $\delta$ are analytic, and assuming that $X(0) \neq 0$, that is, the coefficients $a_{-1,0}$ and $b_{0,-1}$ in the expansion
\begin{equation}\label{eq-constrained-singular-curve-theorem}
\left\{
  \begin{array}{rcl}
    \delta(x,y)\dot{x} & = & a_{-1,0} + a_{0,0}x + a_{-1,1}y + ..., \\
    \delta(x,y)\dot{y} & = & b_{0,-1} + b_{0,0}y + b_{1,-1}x + ...,
  \end{array}
\right.
\end{equation}
satisfy $a_{-1,0}^{2} + b_{0,-1}^{2} \neq 0$, the Theorem \ref{teo-singular-curves-ade-resolution} gives conditions to assure when the System \eqref{eq-constrained-singular-curve-theorem} is orientation preserving $C^{0}$-orbitally equivalent to the system
\begin{equation}\label{eq-constrained-singular-curve-theorem-constante-case}
\left\{
  \begin{array}{rcl}
    \delta(x,y)\dot{x} & = & a_{-1,0}, \\
    \delta(x,y)\dot{y} & = & b_{0,-1},
  \end{array}
\right.
\end{equation}
which belongs to $\mathcal{A}$.

The proof of Theorem \ref{teo-singular-curves-ade-resolution} is given by straightforward computations. Indeed, one must compare the process of resolution of singularities of both \eqref{eq-constrained-singular-curve-theorem} and \eqref{eq-constrained-singular-curve-theorem-constante-case} in order to establish conditions to the coefficients of the Taylor expansion in such a way that, after a suitable finite sequence of weighted blow ups, the systems \eqref{eq-constrained-singular-curve-theorem} and \eqref{eq-constrained-singular-curve-theorem-constante-case} have the same elementary singularity scheme. System \eqref{eq-constrained-singular-curve-theorem} must satisfy one of the conditions presented in Table \ref{tab-conditions-teo-c}, where $n_{0}$ is the first positive integer such that $a_{-1,n_{0}}\neq0$. Geometrically, the conditions below mean that we must avoid tangencies between the adjoint vector field of \eqref{eq-constrained-singular-curve-theorem} and the components of the impasse set.

\begin{table}[ht]
\centering
\begin{tabular}{|c|l|}
\hline
Curve $\delta$           & \multicolumn{1}{c|}{Conditions}                                                                                                      \\ \hline
\multirow{3}{*}{$A_{k}$} & Case $1$: \ $k > 1$ and $a_{-1,0} \neq 0$                                                                                                        \\ \cline{2-2}
                         & Case $2$: \ $k = 1$ and $0 \neq a_{-1,0} \neq \pm b_{0,-1}$                                                                                      \\ \cline{2-2}
                         & Case $3$: \ $a_{-1,0} = 0$ and $k - 1 < 2n_{0}$                                                                                                  \\ \hline
\multirow{3}{*}{$D_{k}$} & Case $4$: \ $a_{-1,0} \neq 0$ and $b_{m_{0},-1} = 0$ for all integer $m_{0}\geq0$                                                                \\ \cline{2-2}
                         & Case $5$:  \ $k = 4$ and $0\neq a_{-1,0} \neq \pm b_{0,-1}$                                                                                       \\ \cline{2-2}
                         & Case $6$:  \ $a_{-1,0} = 0$ and $k - 4 < 2n_{0}$                                                                                                  \\ \hline
\multirow{2}{*}{$E_{7}$} & Case $7$:  \ $b_{0,-1} \neq 0$                                                                                                                    \\ \cline{2-2}
                         & Case $8$: \ $a_{-1,0} \neq 0$ and $b_{m_{0},-1} = 0$ for all integer $m_{0}\geq0$                                                                \\ \hline
$E_{6}$ or $E_{8}$       & \begin{tabular}[c]{@{}l@{}}without requiring further assumptions in the\\ Taylor expansion of the adjoint vector field.\end{tabular} \\ \hline
\end{tabular}
\caption{The constrained system \eqref{eq-constrained-singular-curve-theorem} must satisfy one of the conditions of this table in order to be equivalent to a system in $\mathcal{A}$. The number $n_{0}$ is the first positive integer such that $a_{-1,n_{0}}\neq0$.}\label{tab-conditions-teo-c}
\end{table}

In what follows we briefly describe the phase portraits of \eqref{eq-constrained-singular-curve-theorem-constante-case} near the exceptional divisor for each curve $\delta$ of Table \ref{tab-arnold-ade}. We start this section reviewing the construction of the Newton polygon and how we can define such mathematical object for planar constrained systems.

\subsection{The Newton polygon}
\noindent

The construction of the so called Newton polygon associated to a planar analytic vector field is well-known in the literature, and we refer to \cite{AlvarezFerragutJarque, DumortierLlibreArtes, Pelletier} for details. Afterwards we will see how to define such polygon for constrained systems.

Let $X$ be an analytic vector field. Just as in \cite{Panazzolo}, we write $X$ in the so called logarithmic basis. More precisely, consider
$$X(x,y) = P(x,y)\displaystyle\frac{\partial}{\partial x} + Q(x,y)\displaystyle\frac{\partial}{\partial y},$$
where
$$P(x,y) = \Big{(}\sum a_{m,n}x^{m}y^{n}\Big{)}x, \ Q(x,y) = \Big{(}\sum b_{m,n}x^{m}y^{n}\Big{)}y,$$
with $a_{m,n},b_{m,n}\in\mathbb{R}$ and $m,n\in \mathbb{Z}$ satisfying:
\begin{enumerate}
  \item For $m < -1$ or $n \leq -1$, $a_{m,n} = 0$;
  \item For $m \leq -1$ or $n < -1$, $b_{m,n} = 0$.
\end{enumerate}

Let $\omega = (\omega_{1},\omega_{2})$ be a vector of positive integers. One can write the planar vector field $X$ as
\begin{equation}\label{def-graduation}
X(x,y) = \displaystyle\sum_{d = -1}^{\infty}X_{d}^{(\omega_{1},\omega_{2})}(x,y),
\end{equation}
where
\begin{equation}\label{def-d-level-graduation}
X_{d}^{(\omega_{1},\omega_{2})}(x,y) = \displaystyle\sum_{\omega_{1} r + \omega_{2} s = d} x^{r}y^{s}\Big{(}a_{r,s}x\displaystyle\frac{\partial}{\partial x} +  b_{r,s}y\displaystyle\frac{\partial}{\partial y}\Big{)}.
\end{equation}

In other words, the vector field $X$ can be written as a sum of quasi-homogeneous components. We say that \eqref{def-graduation} is a \textbf{$(\omega_{1},\omega_{2})$-graduation of $X$}. Each $X_{d}^{(\omega_{1},\omega_{2})}$ is called \textbf{$d$-level of the $(\omega_{1},\omega_{2})$-graduation}.

Given a $(\omega_{1},\omega_{2})$-graduation of $X$, we associate the monomials $a_{r,s}x^{r}y^{s}$ and $b_{r,s}x^{r}y^{s}$ with nonzero coefficients to a point $(r,s)$ in the plane of powers, where each point $(r,s)$ is contained in a line of the form $\{\omega_{1} r + \omega_{2} s = d\}$.

\begin{definition}
The \textbf{support $\mathcal{Q}$ of $X$} is the set
$$\mathcal{Q} = \{(r,s)\in \mathbb{Z}^{2}: a_{r,s}^{2} + b_{r,s}^{2} \neq 0\}.$$
\end{definition}

\begin{definition}\label{def-newton-polygon}
The \textbf{Newton polygon $\mathcal{P}$} associated to the analytic vector field $X$ is the convex envelope of the set $\mathcal{Q} + \mathbb{R}^{2}_{+}$.
\end{definition}

The boundary $\partial\mathcal{P}$ of the Newton polygon $\mathcal{P}$ is the union of a finite number of segments. We enumerate them from the left to the right: $\gamma_{0}, \gamma_{1}, ..., \gamma_{n+1}$. Observe that $\gamma_{0}$ is vertical and $\gamma_{n+1}$ is horizontal. Analogously, the non-regular points of $\partial\mathcal{P}$ are enumerated from the left to the right: $v_{0}, ..., v_{n}$.

\begin{definition}
We say that $v_{0}, ..., v_{n}$ are the \textbf{vertices of the Newton polygon}. The vertex $v_{0} = (r_{0}, s_{0})$ is called \textbf{main vertex} and the segment $\gamma_{1}$ will be called \textbf{main segment}. The number $s_{0}$ is the \textbf{height of the Newton polygon}. See Figure \ref{fig-def-newton-polygon}.
\end{definition}

\begin{figure}[h]
  \center{\includegraphics[width=0.30\textwidth]{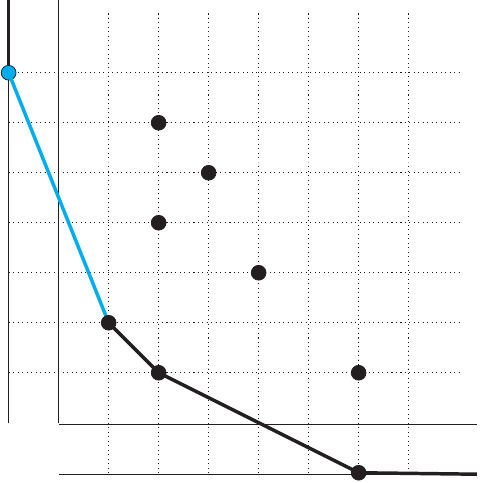}}\\
  \caption{Newton polygon of a vector field. The main vertex and main segment are highlighted in blue.}\label{fig-def-newton-polygon}
\end{figure}

The main segment $\gamma_{1}$ is contained in an affine line of the form $\{r\omega_{1} + s\omega_{2} = R\}$. Observe that the vector $\omega = (\omega_{1},\omega_{2})$ is normal with respect to $\{r\omega_{1} + s\omega_{2} = R\}$. In the resolution of singularities of analytic vector fields, one chooses the vector $\omega$ as the weight vector (see \cite{AlvarezFerragutJarque,Pelletier}). We remark that the Newton polygon strongly depends on the coordinate system adopted.

\subsection{The auxiliary vector field}
\noindent

The next goal is to relate a given constrained system to a Newton polygon. In order to achieve this objective, firstly we will define the so called auxiliary vector field. Recall that a diagonalized constrained system is written in the form
\begin{equation}\label{eq-impasse}
    \delta(x,y)\dot{x} = P(x,y), \ \delta(x,y)\dot{y} = Q(x,y),
\end{equation}
with $\delta (x,y)$ being an irreducible real-analytic function that can be written as
$$\delta(x,y) = \sum c_{k,l}x^{k}y^{l},$$
where $c_{k,l}\in\mathbb{Z}$ is such that $c_{k,l} = 0$ when $k < 0$ or $l<0$.

\begin{definition} The \textbf{auxiliary vector field} $X_{A}$ associated to the constrained differential equation \eqref{eq-impasse} is the vector field
\begin{equation}\label{def-aux-vec-field}
\scriptstyle X_{A}(x,y) = \delta(x,y)\Bigg{(}P(x,y)\frac{\partial}{\partial x} + Q(x,y)\frac{\partial}{\partial y}\Bigg{)} = \Bigg{(}\sum c_{k,l}x^{k}y^{l}\Bigg{)}\Bigg{(}\sum x^{m}y^{n}\big{(}a_{m,n}x\frac{\partial}{\partial x} + b_{m,n}y\frac{\partial}{\partial y}\big{)}\Bigg{)},
\end{equation}
which is real analytic.
\end{definition}

It is easy to sketch the phase portrait of the auxiliary vector field. Outside the impasse set, the auxiliary vector field \eqref{def-aux-vec-field} is obtained by multiplying the constrained differential system \eqref{eq-impasse} by the positive function $\delta^{2}$. On the other hand, the impasse set $\Delta$ is a curve of equilibrium points for \eqref{def-aux-vec-field}.

Since the auxiliary vector field is analytic, all previous definitions and remarks concerning the Newton polygon remain true for \eqref{def-aux-vec-field}. However, observe that the points on the plane of powers are of the form $(k+m, l + n)$, and therefore the support $\mathcal{Q}$ takes form
$$\mathcal{Q} = \{(k+m,l+n)\in \mathbb{Z}^{2}: c_{k,l}(a_{m,n}^{2} + b_{m,n}^{2}) \neq 0\}.$$

In other words, the support of the auxiliary vector field is obtained by the Minkowski Sum \cite{Schneider} between the points of the support of the adjoint vector field $X$ and the points of the support of $\delta$. Moreover, the levels of a $(\omega_{1},\omega_{2})$-graduation are written in the form
$$X_{A,d}^{(\omega_{1},\omega_{2})}= \Bigg{(} \displaystyle\sum_{\omega_{1} k + \omega_{2} l = d_{1}} c_{k,l}x^{k}y^{l}  \Bigg{)} \Bigg{(}   \displaystyle\sum_{\omega_{1} m + \omega_{2}n = d_{2}} x^{m}y^{n}\Big{(}a_{m,n}x\displaystyle\frac{\partial}{\partial x} +  b_{m,n}y\displaystyle\frac{\partial}{\partial y}\Big{)}\Bigg{)},$$
where $d_{1} + d_{2} = d$. Finally, we remark that the Newton polygon of the adjoint vector field and the auxiliary vector field are not necessarily the same.

The process of resolution of singularities of real analytic 2-dimensional constrained differential systems with weighted blow-ups was discussed in details in \cite{PerezSilva}, where one can also find examples. Such process can be summarized as follows. Given a constrained system, we write the system in its diagonalized form and then we define the auxiliary vector field $X_{A}$. The auxiliary vector field is an analytic vector field and it allows us to apply well known techniques of the literature, using the Newton polygon $\mathcal{P}_{X_{A}}$ to chose the weight vector of the blow up.

Now we are able to study constrained systems in the set $\mathcal{A}$.

\subsection{Curve $A_{k}$, $k\geq 1$}

\noindent

Firstly assume the conditions in the case 1. After a weighted blow up in the $x$ direction, the origin $0\in\mathbb{E}_{x}$ is a hyperbolic saddle and the impasse curve satisfies $1 \pm \tilde{y}^{k+1} = 0$. In the $y$ direction there are no equilibrium points in $\mathbb{E}_{y}$ and the impasse curve satisfies the equation $\tilde{x}^{2} \pm 1 = 0$. This case is equivalente to the case 2. See Figures \ref{fig-ak-blow-up-1-a} and \ref{fig-ak-blow-up-2-a}.

Now, consider the case 3. After a weighted blow-up in the $x$ direction, there are no equilibrium points in $\mathbb{E}_{x}$, and in the $y$ direction the origin $0\in\mathbb{E}_{y}$ is a hyperbolic saddle. The impasse set behaves just as the previous cases.

Geometrically, $n_{0}$ measures the tangency order between the adjoint vector field $X$ and the coordinate axis $\{x = 0\}$, and $k$ measures the tangency order between $\delta$ and $\{x = 0\}$. Therefore, the assumption $k-1 < 2n_{0}$ sets a relation between such tangency orders. It is important to remark that, when $k-1 \geq 2n_{0}$, the main segment of the Newton polygon of the auxiliary vector field of \eqref{eq-constrained-singular-curve-theorem} has the point $(1,n_{0})$, and such point does not appear in the Newton polygon of the auxiliary vector field of \eqref{eq-constrained-singular-curve-theorem-constante-case} See figures \ref{fig-ak-blow-up-1-b} and \ref{fig-ak-blow-up-2-b}.

\begin{figure}[h!]
  \center{\includegraphics[width=0.90\textwidth]{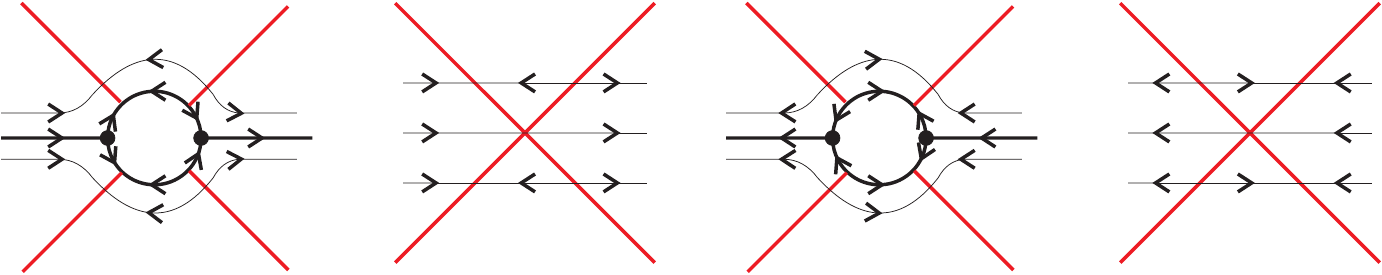}}\\
  \caption{Dynamics when $a_{-1,0}\neq0$, for $\delta$ $A_{k}$-type, $k$ odd.}\label{fig-ak-blow-up-1-a}
\end{figure}

\begin{figure}[h!]
  \center{\includegraphics[width=0.90\textwidth]{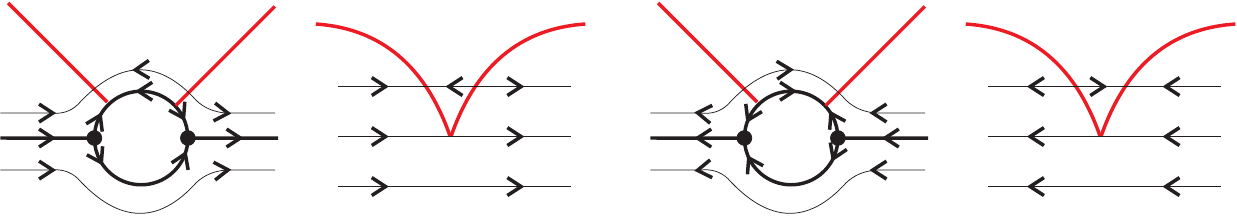}}\\
  \caption{Dynamics when $a_{-1,0}\neq0$, for $\delta$ $A_{k}$-type, $k$ even.}\label{fig-ak-blow-up-2-a}
\end{figure}

\begin{figure}[h!]
  \center{\includegraphics[width=0.90\textwidth]{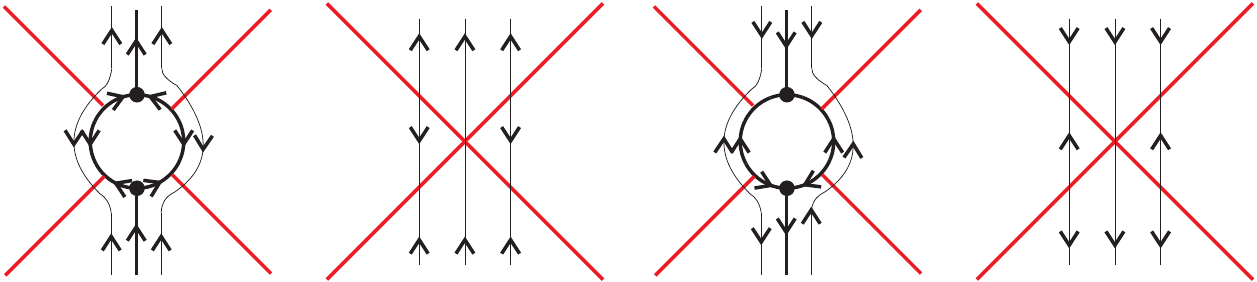}}\\
  \caption{Dynamics when $a_{-1,0} = 0$, for $\delta$ $A_{k}$-type, $k$ odd.}\label{fig-ak-blow-up-1-b}
\end{figure}

\begin{figure}[h!]
  \center{\includegraphics[width=0.90\textwidth]{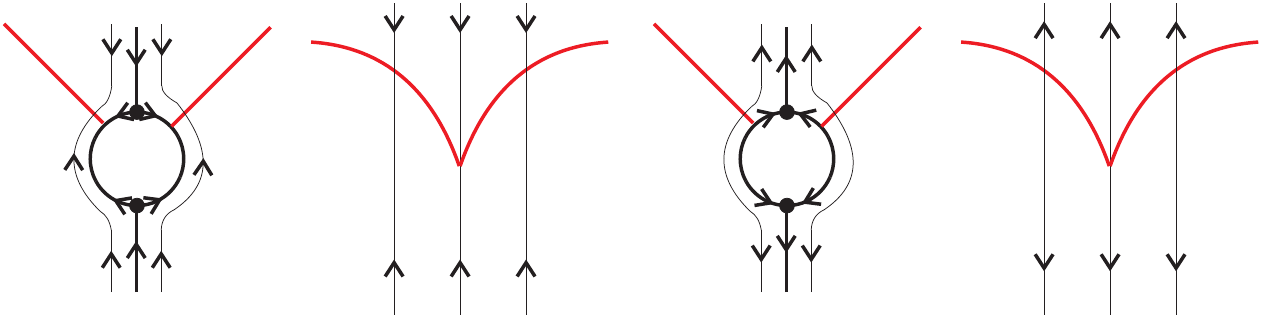}}\\
  \caption{Dynamics when $a_{-1,0} = 0$, for $\delta$ $A_{k}$-type, $k$ even.}\label{fig-ak-blow-up-2-b}
\end{figure}

\subsection{Curve $D_{k}$, $k\geq 4$}

\noindent

Assume the hypotheses in case 4. After a weighted blow up in the $x$ direction, the origin $0\in\mathbb{E}_{x}$ is a saddle point of the adjoint vector field $X$ and there is a separatrix that coincides with a component of the impasse set $\{\tilde{y} \pm \tilde{y}^{k-1} = 0\}$. In the $y$ direction, there are no equilibrium points in $\mathbb{E}_{y}$ and the impasse set is give by $\{\tilde{x}^{2} \pm 1 = 0\}$.

We emphasize that the condition $b_{m_{0},-1} = 0$ for all integer $m_{0}\geq 0$ avoids tangency points between the adjoint vector field and the component $\{y = 0\}$ of the impasse set, in which it would leads us to a different resolution of singularities from system \eqref{eq-constrained-singular-curve-theorem-constante-case}. See Figures \ref{fig-dk-blow-up-3-a} and \ref{fig-dk-blow-up-3-b}.

Now assume the hypotheses of the case 5. In the $y$ direction, there are no equilibrium points of $X$ in $\mathbb{E}_{y}$. On the other hand, in the $x$ direction the origin $0\in\mathbb{E}_{x}$ is a hyperbolic saddle of the adjoint vector field and an impasse point at the same time. Therefore, we must blow-up the origin once again. See Figures \ref{fig-dk-blow-up-1-a} and \ref{fig-dk-blow-up-2-a}

Finally, assume the hypotheses of the case 6. In the $x$ direction the impasse curve is given by $\tilde{y} \pm \tilde{y}^{k-1} = 0$ and there are no equilibrium points of the adjoint vector field in $\mathbb{E}_{x}$. On the other hand, after a weighted blow up in the $y$ direction the impasse curve is given by $\tilde{x}^{2} \pm 1 = 0$ and the origin $0\in\mathbb{E}_{y}$ is a hyperbolic saddle.

The number $n_{0}$ is related to the tangency order between the adjoint vector field $X$ and the axis $\{x = 0\}$, and $k$ is related to the tangency order between $\delta$ and the axis $\{x = 0\}$. Thus the assumption $k-4 < 2n_{0}$ sets a relation between such tangency orders. Observe that in the case $k-4 \geq 2n_{0}$, the main segment of the Newton polygon of the auxiliary vector field of \eqref{eq-constrained-singular-curve-theorem} has the point $(1,n_{0}+1)$, and such point obviously does not appear in the Newton polygon of the auxiliary vector field of \eqref{eq-constrained-singular-curve-theorem-constante-case}. See Figures \ref{fig-dk-blow-up-1-b} and \ref{fig-dk-blow-up-2-b}.

\begin{figure}[h!]
  \center{\includegraphics[width=0.90\textwidth]{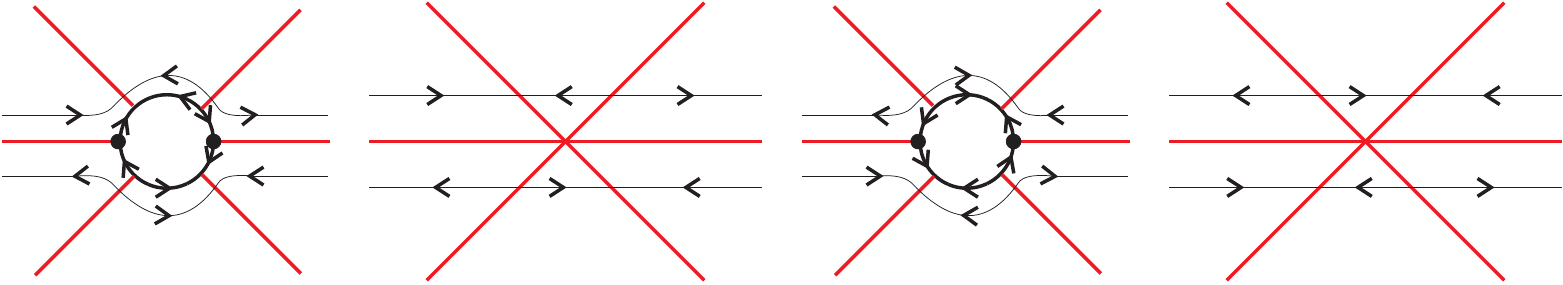}}\\
  \caption{Dynamics when $a_{-1,0} = 0$, for $\delta$ $D_{k}$-type, $k$ even.}\label{fig-dk-blow-up-3-a}
\end{figure}

\begin{figure}[h!]
  \center{\includegraphics[width=0.90\textwidth]{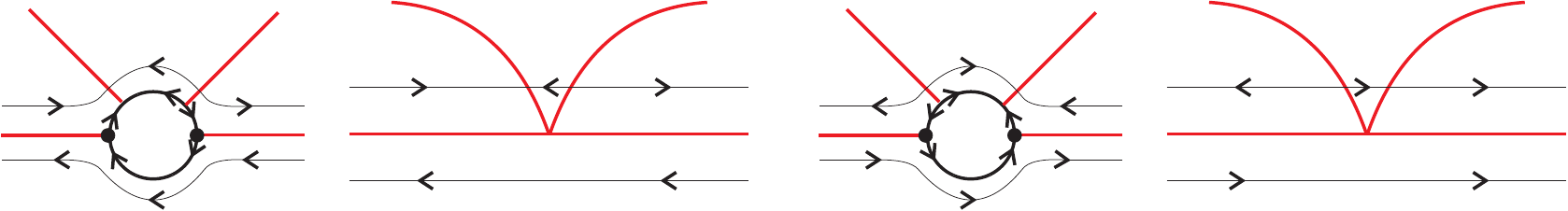}}\\
  \caption{Dynamics when $a_{-1,0} = 0$, for $\delta$ $D_{k}$-type, $k$ odd.}\label{fig-dk-blow-up-3-b}
\end{figure}

\begin{figure}[h!]
  \center{\includegraphics[width=0.90\textwidth]{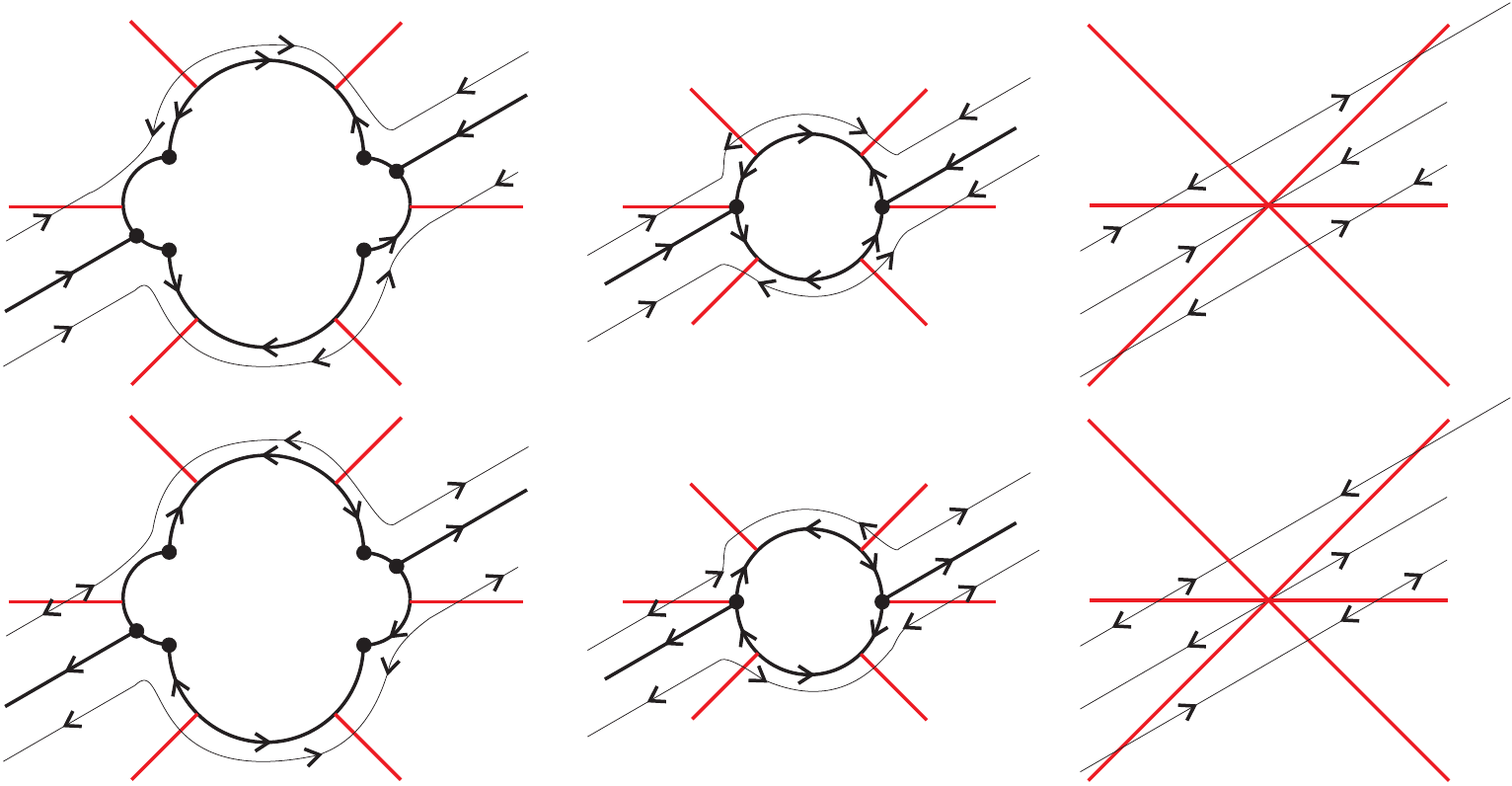}}\\
  \caption{Dynamics when $a_{-1,0} \neq 0$, for $\delta$ $D_{k}$-type, $k$ even.}\label{fig-dk-blow-up-1-a}
\end{figure}

\begin{figure}[h!]
  \center{\includegraphics[width=0.90\textwidth]{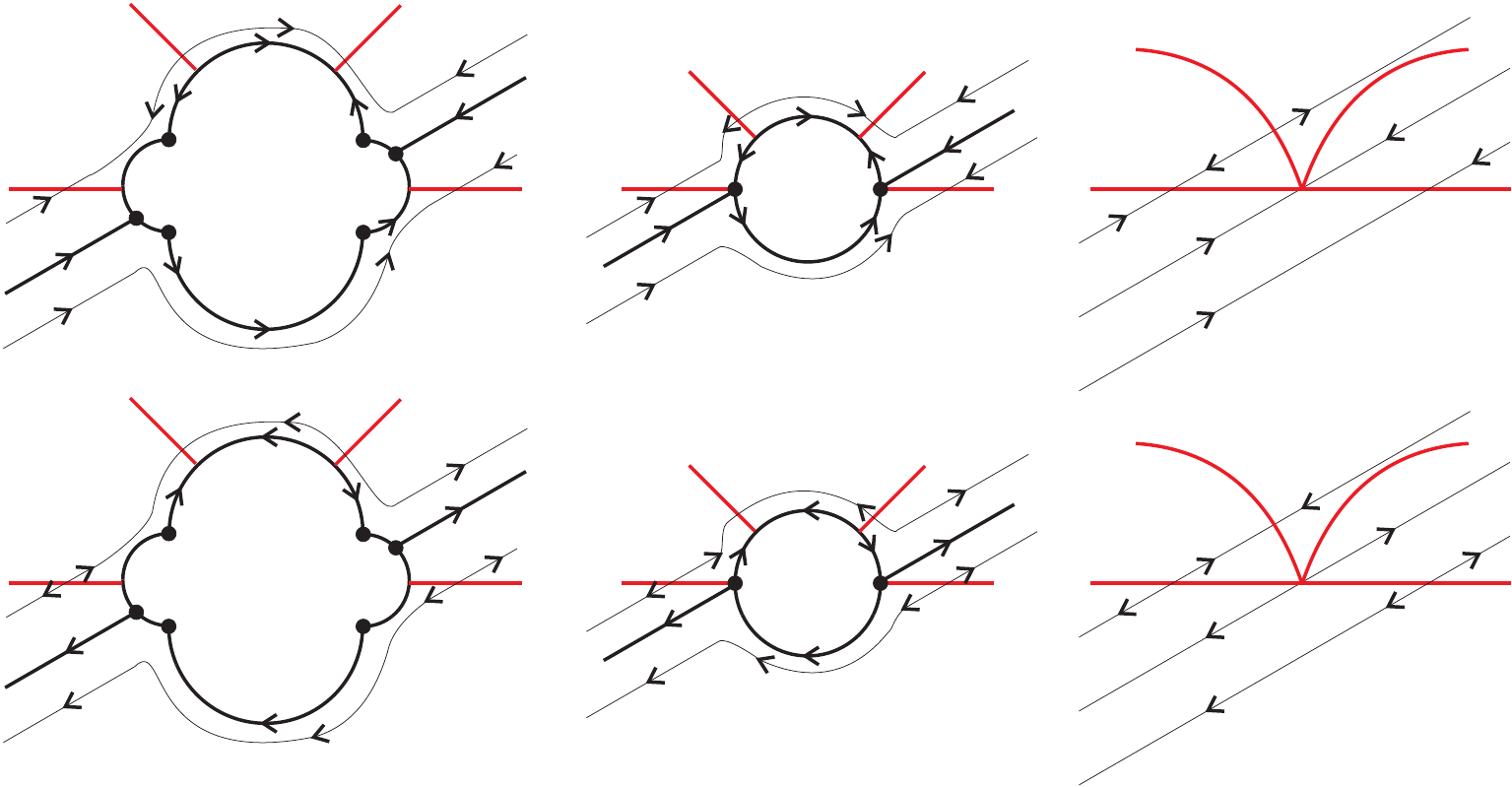}}\\
  \caption{Dynamics when $a_{-1,0} \neq 0$, for $\delta$ $D_{k}$-type, $k$ odd.}\label{fig-dk-blow-up-2-a}
\end{figure}

\begin{figure}[h!]
  \center{\includegraphics[width=0.90\textwidth]{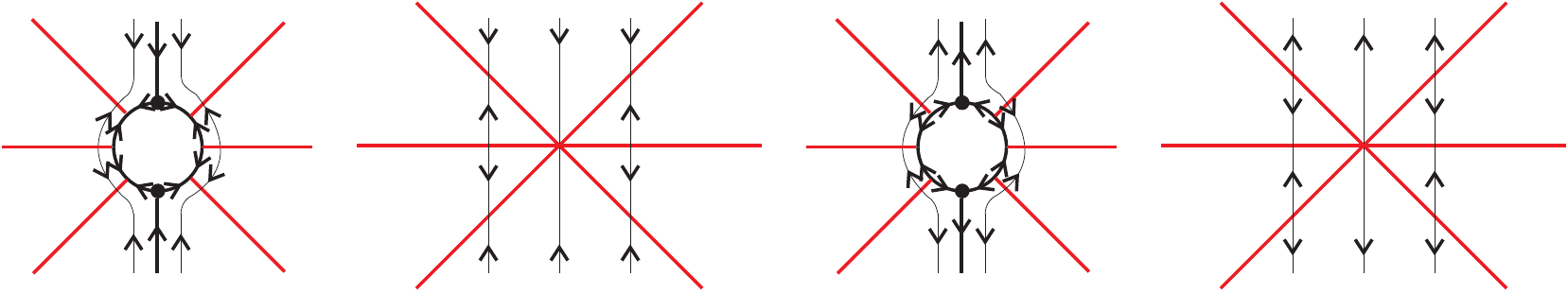}}\\
  \caption{Dynamics when $a_{-1,0} = 0$, for $\delta$ $D_{k}$-type, $k$ even.}\label{fig-dk-blow-up-1-b}
\end{figure}

\begin{figure}[h!]
  \center{\includegraphics[width=0.90\textwidth]{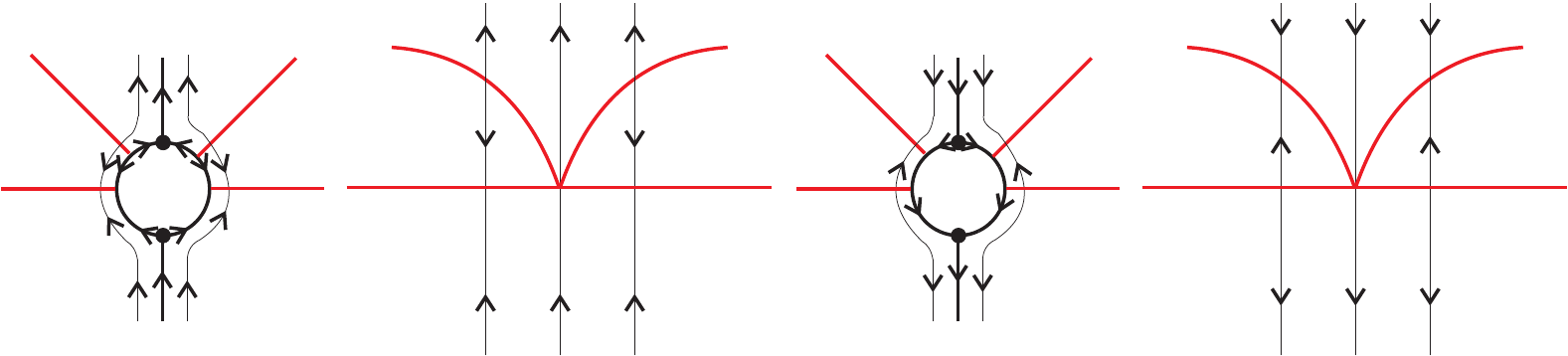}}\\
  \caption{Dynamics when $a_{-1,0} = 0$, for $\delta$ $D_{k}$-type, $k$ odd.}\label{fig-dk-blow-up-2-b}
\end{figure}

\subsection{Curve $E_{6}$}

\noindent

Suppose without loss of generality that $\delta = x^{3} - y^{4}$. In the positive $x$ direction the impasse curve intersects the exceptional divisor at the points $(0,\pm 1)$. For $a_{-1,0}\neq 0$ the origin $0\in\mathbb{E}_{x}$ is a hyperbolic saddle and if $a_{-1,0} = 0$ we do not have equilibrium points of the adjoint vector field on the divisor. In the positive $y$ direction, the impasse curve intersects the exceptional divisor at $(1,0)$. If $a_{-1,0}\neq 0$, there are no equilibrium points of the adjoint vector field in $\mathbb{E}_{y}$ and if $a_{-1,0} =  0$ the origin is a hyperbolic saddle. See Figures \ref{fig-e6-blow-up} and \ref{fig-e6-blow-up-2}.

\begin{figure}[h!]
  \center{\includegraphics[width=0.90\textwidth]{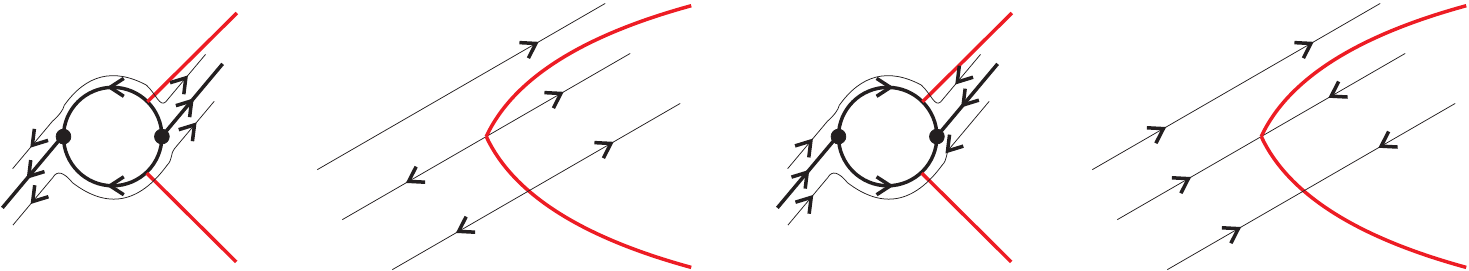}}\\
  \caption{Dynamics when $a_{-1,0}\neq0$, for $\delta$ $E_{6}$-type.}\label{fig-e6-blow-up}
\end{figure}

\begin{figure}[h!]
  \center{\includegraphics[width=0.80\textwidth]{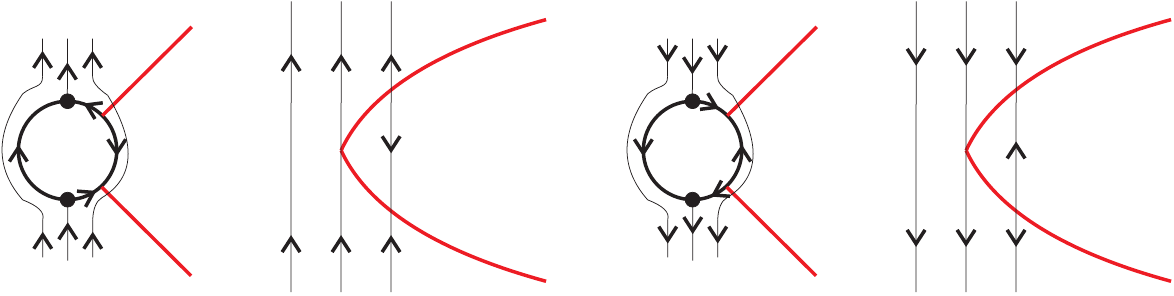}}\\
  \caption{Dynamics when $a_{-1,0} = 0$, for $\delta$ $E_{6}$-type.}\label{fig-e6-blow-up-2}
\end{figure}

\subsection{Curve $E_{7}$}

\noindent

Suppose that $a_{-1,0} \neq 0$ and $b_{m_{0},-1} = 0$, for all integer $m_{0}\geq0$. In the positive $x$ direction, the impasse curve intersects the $\mathbb{E}_{x}$ at the points $(0,-1)$, $(0,0)$ and $(0,1)$. Moreover, the origin is a saddle point of the adjoint vector field and there is a separatrix that coincides with a component of the impasse set. In the positive $y$ direction, there are no equilibrium points of $X$ in $\mathbb{E}_{y}$ and the impasse set intersects the exceptional divisor at $(1,0)$.

Once again we remark that the condition $b_{m_{0},-1} = 0$ avoids tangency points between the adjoint vector field $X$ and the component $\{y = 0\}$ of the impasse set. Such tangency points would lead us to a different resolution of singularities.

If we consider $b_{0,-1} \neq 0$, in the positive $x$ direction there are no equilibrium points of $X$ in $\mathbb{E}_{x}$, and in the positive $y$ direction the origin $0\in\mathbb{E}_{y}$ is a saddle points of $X$. The impasse curve behaves just as in the previous case.

See Figures \ref{fig-e7-blow-up-2} and \ref{fig-e7-blow-up}.

\begin{figure}[h!]
  \center{\includegraphics[width=0.95\textwidth]{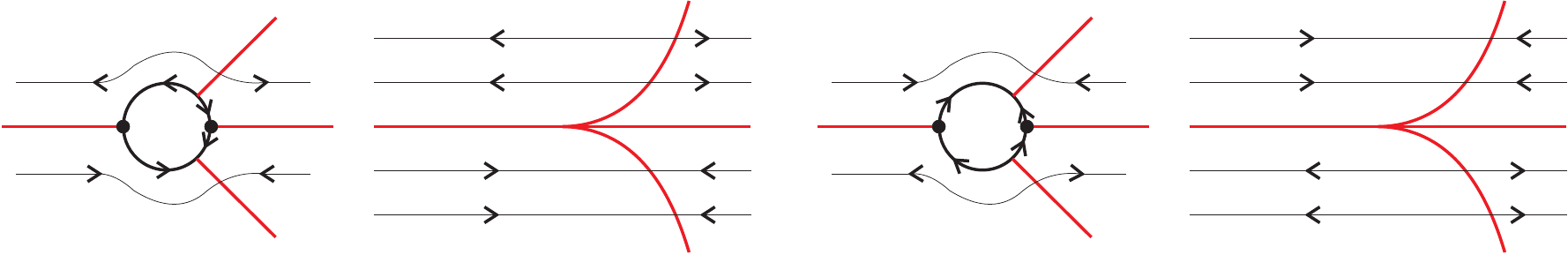}}\\
  \caption{Dynamics when $b_{m_{0},-1} = 0$ for all $m_{0}\geq 0$, where $\delta$ is $E_{7}$-type.}\label{fig-e7-blow-up-2}
\end{figure}

\begin{figure}[h!]
  \center{\includegraphics[width=0.95\textwidth]{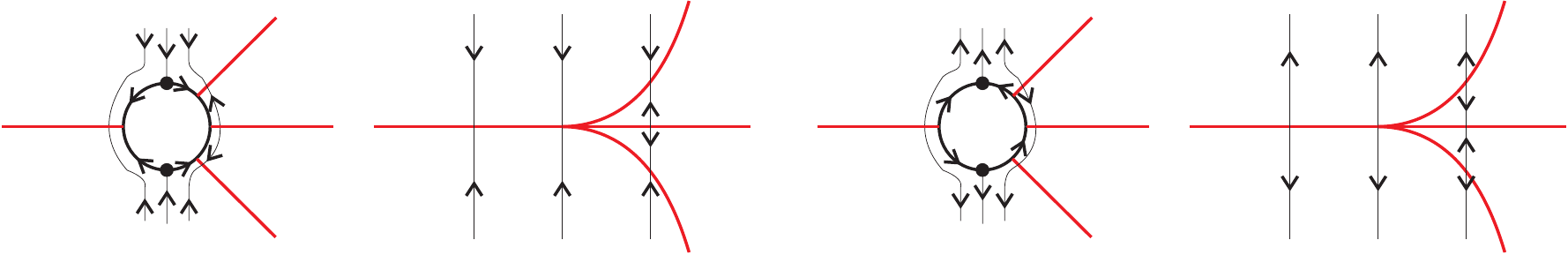}}\\
  \caption{Dynamics when $b_{0,-1}\neq0$, for $\delta$ $E_{7}$-type.}\label{fig-e7-blow-up}
\end{figure}

\subsection{Curve $E_{8}$}

\noindent

In the positive $x$ direction, the impasse curve intersects the exceptional divisor at $(0,-1)$. If $a_{-1,0}\neq 0$, then the origin $0\in\mathbb{E}_{x}$ is a hyperbolic saddle of the adjoint vector field and if $a_{-1,0} = 0$ there are no equilibrium equilibrium points of the adjoint vector field on the divisor. Finally, in the positive $y$ direction the impasse curve intersects the exceptional divisor $\mathbb{E}_{y}$ at $(-1,0)$. If $a_{-1,0}\neq 0$, there are no equilibrium points of the adjoint vector field in $\mathbb{E}_{y}$ and if $a_{-1,0} = 0$ the origin is a hyperbolic saddle of the adjoint vector field. See Figures \ref{fig-e8-blow-up} and \ref{fig-e8-blow-up-2}.

\begin{figure}[h!]
  \center{\includegraphics[width=0.90\textwidth]{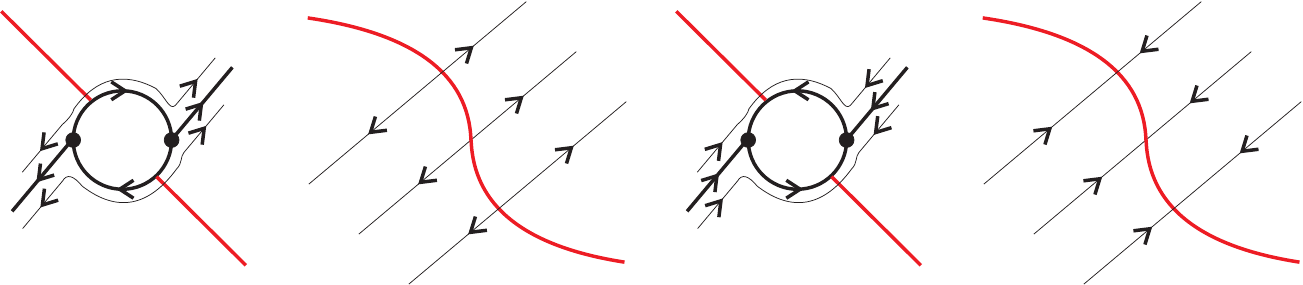}}\\
  \caption{Dynamics when $a_{-1,0}\neq0$, for $\delta$ $E_{8}$-type.}\label{fig-e8-blow-up}
\end{figure}

\begin{figure}[h!]
  \center{\includegraphics[width=0.90\textwidth]{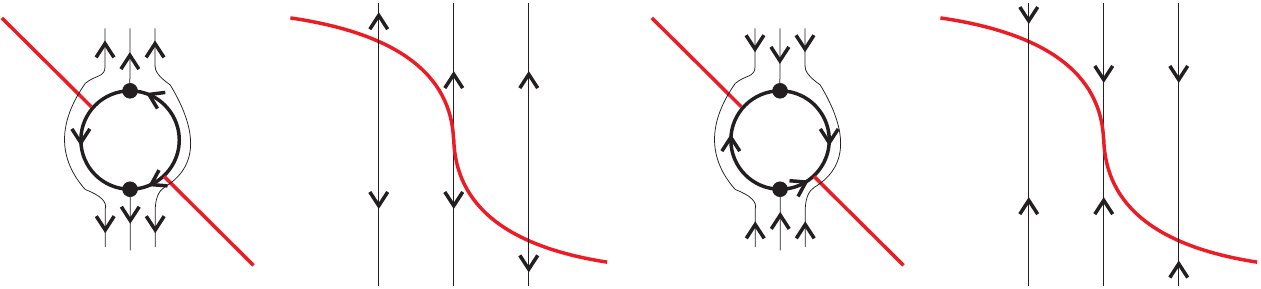}}\\
  \caption{Dynamics when $a_{-1,0} = 0$, for $\delta$ $E_{8}$-type.}\label{fig-e8-blow-up-2}
\end{figure}

\section{Topological determination of a constrained system with smooth impasse curve via Newton polygon}\label{sec-newton-pol}
\noindent

Firstly, let us recall a classical result of the literature. Let $X$ be a planar analytic vector field defined near the origin such that $X(0) = 0$, and denote its Newton polygon by $\mathcal{P}_{X}$ (see Definition \ref{def-newton-polygon} below). In \cite{BrunellaMiari} the authors proved that, under some non degeneracy conditions, the terms of $X$ associated to the the boundary $\partial\mathcal{P}_{X}$ of the Newton polygon determines the phase portrait of $X$ near the equilibrium point.

More precisely, given an analytic vector field $X$, the terms of $X$ associated to the points in $\partial\mathcal{P}_{X}$ define the so called principal part of $X$, which is denoted by $X_{\Omega}$. The Theorem A of \cite{BrunellaMiari} assures that $X$ and $X_{\Omega}$ are topologically equivalent.

Here we present a version of such result for analytic planar constrained differential systems near a non-singular point of the impasse curve. Our strategy is to work in a convenient coordinate system, in such a way that the classical result presented in \cite{BrunellaMiari} can be applied during the resolution of singularities.

During the resolution of singularities, one must apply a finite number of operations (weighted blow-ups) in order to obtain a ``simpler'' constrained system. This notion of ``simple'' comes from the notion of elementary constrained system (see Definition \ref{def-elementary-points}). In what follows we recall how to identify elementary points of the constrained system by means of the Newton polygon of the auxiliary vector field. We refer \cite{PerezSilva} for details.

\begin{definition}\label{def-desing1}
A planar constrained differential system is \textbf{Newton elementary at $p$} if the Newton polygon $\mathcal{P}$ associated to the auxiliary vector field $X_{A}$ satisfies one of the following:
\begin{itemize}
  \item The main vertex of $\mathcal{P}$ is $(0,0)$, $(0,-1)$ or $(-1,0)$ (that is, if the height is less or equal to zero);
  \item The main segment $\gamma_{1}$ is horizontal.
\end{itemize}
\end{definition}

\begin{theorem}\label{teo-eq-definitions} (See \cite{PerezSilva})
A planar constrained system is elementary at $p$ if, and only if, it is Newton elementary at $p$.
\end{theorem}

\subsection{Favorable coordinates}
\noindent

Since our study is local, for simplicity sake we denote a planar real analytic constrained differential system defined near the origin as $(X, \delta)$, where $X = \Big{(}P(x,y),Q(x,y)\Big{)}$ is the adjoint vector field and $\delta$ is a real irreducible analytic function. As usual, the auxiliary vector field will be denoted by $X_{A}$. The Newton polygons of the adjoint vector field and the auxiliary vector field will be denote by $\mathcal{P}_{X}$ and $\mathcal{P}_{X_{A}}$, respectively.

In this section, constrained differential systems with regular impasse set are considered. By ``regular impasse set'' we mean that the gradient vector $\nabla \delta$ is nonzero. The main goal is to extend the Theorem A of \cite{BrunellaMiari} to the context of constrained differential systems, that is, the objective is to show that the terms associated to the boundary $\partial\mathcal{P}_{X_{A}}$ of the Newton Polygon of the auxiliary vector field determine (under topological equivalence) the phase portrait near the singularity of the constrained system. In order to achieve this objetive we will introduce a convenient coordinate system.

\begin{definition}
The Newton polygon $\mathcal{P}$ associated to a vector field is \textbf{controllable} if the main vertex $v_{0}$ is contained in $\{0\}\times\mathbb{Z}$ or $\{-1\}\times\mathbb{Z}$.
\end{definition}

By Lemma 9 in \cite{PerezSilva}, we can always assume without loss of generality that the Newton polygon of an analytic vector field is controllable.

\begin{lemma}\label{lemma-change-coordinates-brunella}
Let $(X, \delta)$ be a constrained system such that the origin is an isolated equilibrium point and the impasse set is regular. Then there is a coordinate system such that:
\begin{itemize}
  \item The impasse set is given by $\Delta = \{y = 0\}$;
  \item The Newton polygons $\mathcal{P}_{X_{A}}$ and $\mathcal{P}_{X}$ of the auxiliary vector field $X_{A}$ and the adjoint vector field $X$, respectively, are controllable;
  \item The polygon $\mathcal{P}_{X_{A}}$ is obtained by increasing one unity on the second coordinate of each point of the polygon $\mathcal{P}_{X}$.
\end{itemize}
\end{lemma}
\begin{proof}
The first item is true due to the Implicit Function Theorem. For the second item, we can apply the change of coordinates
$$x = \bar{x} + \lambda\bar{y}; \ y = \bar{y},$$
with $\lambda\in\mathbb{R}$ (see Lemma 9, \cite{PerezSilva}). Observe that after this change of coordinates, the impasse set is still of the form $\Delta = \{y = 0\}$. Recall that the Newton polygon $\mathcal{P}_{X_{A}}$ of the auxiliary vector field $X_{A}$ is obtained by the Minkowski sum between the Newton polygons of $\delta$ and $X$. Since in this coordinate system the function $\delta$ is given by $\delta = y$, it follows that the Newton polygon $\mathcal{P}_{X_{A}}$ of the auxiliary vector field $X_{A}$ is obtained by increasing one unity the second coordinate of each point of the Newton polygon $\mathcal{P}_{X}$ of $X$. See Figure \ref{fig-brunellamiari-newton-polygon-change}.
\end{proof}

\begin{figure}[h!]
  \center{\includegraphics[width=0.50\textwidth]{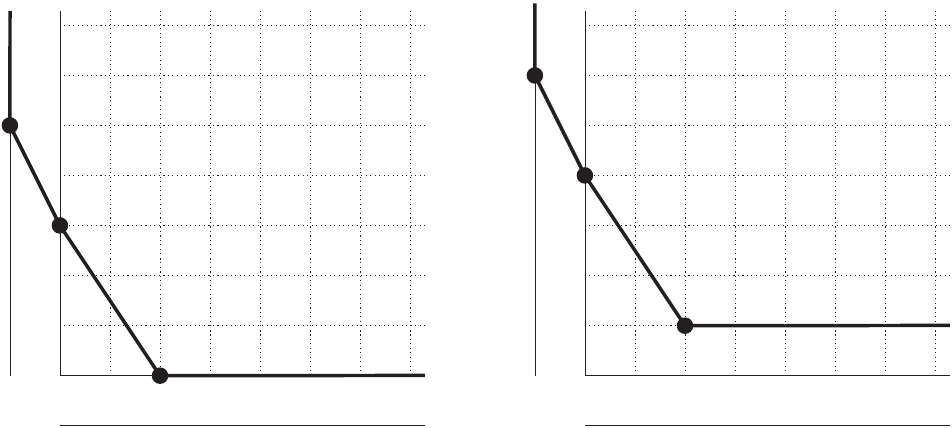}}\\
  \caption{Newton polygons $\mathcal{P}_{X}$ (left) and $\mathcal{P}_{X_{A}}$ (right).}\label{fig-brunellamiari-newton-polygon-change}
\end{figure}

\begin{definition}
The coordinate system given by Lemma \ref{lemma-change-coordinates-brunella} is called \textbf{favorable}.
\end{definition}

Due to Lemma \ref{lemma-change-coordinates-brunella}, without loss of generality we can always take favorable coordinates for $(X,\delta)$. Furthermore, by Lemma \ref{lemma-change-coordinates-brunella}, it follows that $\mathcal{P}_{X_{A}}$ and $\mathcal{P}_{X}$ have the same number of segments. In addition, given a segment $\gamma_{i}\subset\partial\mathcal{P}_{X_{A}}$, the respective segment $\gamma_{i}'\subset\partial\mathcal{P}_{X}$ has the same slope as $\gamma_{i}$. This implies that, if the main segment $\gamma_{1}\subset\partial\mathcal{P}_{X_{A}}$ is contained in a line of the form $\{r\omega_{1} + s\omega_{2} = R\}$, then the segment $\gamma_{1}'\subset\partial\mathcal{P}_{X}$ is contained in a line of the form $\{r\omega_{1} + s\omega_{2} = S\}$. Therefore, at each step of the resolution of singularities, the weights of the blow ups that we use in the desingularization of the constrained system and the vector field $X$ are the same.

Now we are ready to define the principal part of a planar constrained system with regular impasse set.

\begin{definition}
The \textbf{principal part of a 2-dimensional constrained system with regular impasse set} is a pair $(X_{\Omega},\delta)$ in favorable coordinates, where $X_{\Omega}$ is the principal part of the adjoint vector field $X$ in the sense of \cite{BrunellaMiari}.
\end{definition}

We recall from \cite{BrunellaMiari} that the principal part of a real analytic vector field with respect to a fixed system of coordinates is given by
$$X_{\Omega}(x,y) = \displaystyle\sum_{j = 1}^{n}X_{j}(x,y),$$
where $X_{j}$ is the vector field defined by the terms of $X$ that belong to the segment $\gamma_{j}\subset\partial\mathcal{P}_{X}$, where  $\gamma_{j}$ is neither vertical nor horizontal. \\

\noindent \textbf{Example.} Consider the diagonalized constrained differential system
\begin{equation}\label{eq-exe-principal-part}
y\dot{x} = y^{3} + x^{2}y, \ y\dot{y} = xy + x^{4};
\end{equation}
which is already written in favorable coordinates and whose auxiliary vector field is given by
\begin{equation}\label{eq-exe-principal-part-aux}
X_{A}(x,y) = y\big{(}y^{3} + x^{2}y\big{)}\displaystyle\frac{\partial}{\partial x} + y\big{(}xy + x^{4}\big{)}\displaystyle\frac{\partial}{\partial y}.
\end{equation}

The support of \eqref{eq-exe-principal-part-aux} is the set $\mathcal{Q} = \{(-1,4);(1,2);(4,0);(1,1)\}$, and the terms associated to the boundary $\partial\mathcal{P}_{X_{A}}$ of the Newton polygon of $X_{A}$ are $y\big{(}y^{3}\big{)}\displaystyle\frac{\partial}{\partial x}$, $y\big{(}xy\big{)}\displaystyle\frac{\partial}{\partial y}$ and $y\big{(}x^{4}\big{)}\displaystyle\frac{\partial}{\partial y}$. Thus the principal part of \eqref{eq-exe-principal-part} is
\begin{equation}
y\dot{x} = y^{3}, \ y\dot{y} = xy + x^{4}.
\end{equation}

\subsection{Non-degeneracy conditions}
\noindent

In this subsection we discuss some non-degeneracy conditions required for the pair $(X, \delta)$ in order to show the existence of an orbital $C^{0}$-equivalence between $(X,\delta)$ and $(X_{\Omega},\delta)$. As usual, we consider $(X,\delta)$ in favorable coordinates.

The first assumption is that the origin is an isolated equilibrium point of the adjoint vector field $X$. This implies that $\partial\mathcal{P}_{X}$ intersects the coordinate axes of the $rs$-plane. However, we remark that, in general, the Newton polygon $\partial\mathcal{P}_{X_{A}}$ of the auxiliary vector field only intersects the $s$-axis.

The next non-degeneracy condition was introduced in \cite{BrunellaMiari}.

\begin{definition}
Let $X$ be an analytic vector field such that $p$ is an isolated equilibrium point and let $X_{\Omega}$ be its principal part. We say that $X$ is \textbf{Newton non-degenerated at $p$} if any quasi homogeneous component $X_{j}$ of $X_{\Omega}$ associated to a side $\gamma_{j}\subset\partial\mathcal{P}_{X}$ does not have singularities in $(\mathbb{R}\backslash\{0\})^{2}$, that is, if each $X_{j}$ does not have singularities outside the coordinate axes.
\end{definition}

This non-degeneracy condition implies that, during the resolution of singularities of analytic vector fields, the only point in the exceptional divisor that has positive height is the origin in the $x$ direction. In other words, all the equilibrium points in $\Big{(}\mathbb{E}_{x}\backslash\{(0,0)\}\Big{)}\cup\mathbb{E}_{y}$ will always be elementary (see Proposition B, \cite{Pelletier}).

Moreover, such condition depends on the coordinate system adopted, that is, an analytic vector field can be Newton non-degenerated in a fixed coordinate system, but not be in other coordinate system (see \cite{Pelletier}). However, being Newton non-degenerated is a generic property in the set of all analytic vector fields (see Proposition 6, \cite{BrunellaMiari}), and therefore this is a generic condition in the set $\Gamma$ of all planar analytic constrained systems defined near the origin in favorable coordinates.

We end this subsection with the following remark. Suppose that $\mathcal{P}_{X}$ contains a point of the form $(M,-1)$. In other words, suppose that the component $Q$ of $X = (P,Q)$ has a term of the form $b_{M,-1}x^{M}$, with $b_{M,-1} \neq 0$. Geometrically, this condition means that any phase curve of $X$ does not coincide with the impasse curve $\Delta = \{y = 0\}$. By Lemma \ref{lemma-change-coordinates-brunella} it follows that $\mathcal{P}_{X_{A}}$ has a point of the form $(M,0)$. On the other hand, if $\mathcal{P}_{X}$ does not contain a point of the form $(M,-1)$, then the impasse set $\Delta$ coincides with a phase curve. By Lemma \ref{lemma-change-coordinates-brunella} it follows that $\mathcal{P}_{X_{A}}$ does not intersect the $r$-axis of the $rs$-plane.

\subsection{Proof of Theorem \ref{teo-brunellamiari-2}: The topological equivalence between a 2-dimensional constrained system and its principal part}
\noindent

Firstly, we consider the case where the boundary of the Newton polygon has just one side that is neither horizontal nor vertical, and afterwards we generalize such result. From now on, the pair $(X,\delta)$ satisfies the hypotheses discussed in the previous subsection.

\begin{lemma}\label{lemma-case-2-brunella}
Let $(X,\delta)$ be a 2-dimensional constrained system and take favorable coordinates. Suppose that
\begin{enumerate}
  \item The vector field $X$ is Newton non-degenerated at the origin;
  \item $X_{\Omega}$ = $X_{1}$, that is, the boundary $\partial\mathcal{P}_{X}$ of the polygon of $X$ has just one segment that is neither vertical nor horizontal.
\end{enumerate}
Then $(X_{\Omega},\delta)$ and $(X,\delta)$ have the same elementary singularity scheme.
\end{lemma}
\begin{proof}
We will study the case where the main vertex is of the form $(0,N)$. The case where the main vertex is $(-1,N)$ is analogous. The proof is based in directional weighted blow ups and we focus only in the positive $x$ and $y$ directions, provided that the computations are similar in the negative directions.

Since we are adopting favorable coordinates, the impasse curve $\{y = 0\}$ only appears in the $x$ direction. This means that, in the $y$ direction, there are no impasse points in the exceptional divisor after a weighted blow-up, and therefore we are in the classical case of analytic vector fields. It follows that in the $y$ direction both pairs $(X_{\Omega},\delta)$ and $(X,\delta)$ have the same equilibrium points, and the arrangement of such points is the same for both pairs. We emphasize that such equilibrium points in $\mathbb{E}_{y}$ are semi-hyperbolic (and therefore elementary), because the constrained system is Newton non-degenerated.

After a blow up in the $x$ direction, the first graduation of the auxiliary vector field takes the form
$$\widetilde{X}_{A} = \tilde{y}\Bigg{(}   \displaystyle\sum_{n = -1}^{N}\tilde{y}^{n}\Big{(}\frac{a_{m,n}}{\omega_{1}}\tilde{x}\displaystyle\frac{\partial}{\partial \tilde{x}} +  (b_{m,n} - \frac{\omega_{2}}{\omega_{1}}a_{m,n})\tilde{y}\displaystyle\frac{\partial}{\partial \tilde{y}}\Big{)}\Bigg{)}.$$

Observe that the impasse curve is still of the form $\{\tilde{y} = 0\}$ and the equilibrium points of the strict transformed $\widetilde{X}_{\Omega}$ and $\widetilde{X}$ are the same and they are semi-hyperbolic. We remark that this happens independently if the impasse set coincides or does not coincide with a separatrix of an equilibrium point. The difference is that, in the first case, the main segment of the polygon after the blow-up is horizontal, and in the second case, the polygon after the blow-up does not have positive height.

It follows that the pairs $(X_{\Omega},\delta)$ and $(X,\delta)$ have the same elementary singularity scheme.
\end{proof}

It is interesting to remark that in Lemma \ref{lemma-case-2-brunella}, the process of resolution of singularities essentially desingularizes the vector field, since the impasse set is already ``simple''. This idea will be important in the proof of the next proposition.

\begin{proposition}\label{teo-brunellamiari}
Let $(X,\delta)$ be a 2-dimensional constrained system written in favorable coordinates. If the origin is Newton non-degenerated, then the pairs $(X_{\Omega},\delta)$ and $(X,\delta)$ have the same elementary singularity scheme.
\end{proposition}
\begin{proof}

Once again we consider weighted directional blow ups and study the case where the main vertex is of the form $(0,N)$. The case where the main vertex is $(-1,N)$ is analogous. We will compare the process of resolution of singularities of $(X_{\Omega},\delta)$ and $(X,\delta)$. Since the segments of the Newton polygons $\mathcal{P}_{X_{A}}$ and $\mathcal{P}_{X_{\Omega, A}}$ have the same slope, at each step of the process we apply blow-ups with same weights for both constrained systems. The case were the boundary of the Newton polygon has just one segment that is neither vertical nor horizontal was treated in Lemma \ref{lemma-case-2-brunella}, so we will give a proof for the general case.

Just as in Lemma \ref{lemma-case-2-brunella}, the impasse set $\{y = 0\}$ only appears in the $x$ direction. This means that in all points of the divisor $\mathbb{E}_{y}$, both pairs $(X_{\Omega},\delta)$ and $(X,\delta)$ have the same equilibrium points, all of them are semi-hyperbolic and the arrangement of such equilibrium points are the same for both $(X_{\Omega},\delta)$ and $(X,\delta)$, provided that in the positive $y$ direction we are in the classical case and the constrained system is Newton non-degenerated. Therefore it is sufficient to look at the origin of the $x$ direction. See Figure \ref{fig-brunellamiari-direction-y}.

\begin{figure}[h!]
  \center{\includegraphics[width=0.30\textwidth]{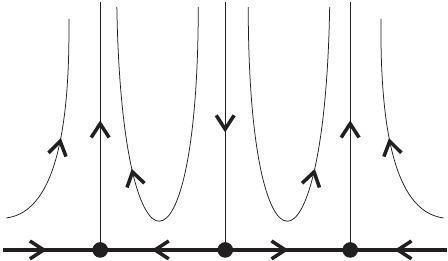}}\\
  \caption{Blow-up in the $y$-direction.}\label{fig-brunellamiari-direction-y}
\end{figure}

Applying the first blow-up with weight $\omega = (\omega_{1}, \omega_{2})$ in the $x$ direction, we obtain the strict transformed $(X^{1}_{\Omega},\delta_{1})$ and $(X^{1},\delta_{1})$. It is straightforward to see that $\delta$ is still of the form $\delta = y$, and therefore the origin is the only impasse point in the divisor. Concerning the adjoint vector field, with similar computations as in Lemma \ref{lemma-case-2-brunella} on the divisor $\mathbb{E}_{x}$ we have the same equilibrium points for both $X$ and $X_{\Omega}$. We emphasize that all the points in $\mathbb{E}_{x}\backslash\{(0,0)\}$ are elementary. Moreover, the arrangement of such equilibrium points in the divisor are the same. Observe that the origin is an equilibrium point and an impasse point at the same time.

\begin{figure}[h!]
  \center{\includegraphics[width=0.15\textwidth]{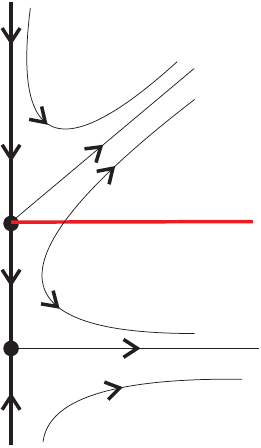}}\\
  \caption{Blow-up in the $x$-direction.}\label{fig-brunellamiari-direction-x}
\end{figure}

Due to Lemma \ref{lemma-change-coordinates-brunella}, the Newton polygons of $(X^{1}_{\Omega},\delta_{1})$ and $(X^{1},\delta_{1})$ are obtained by increasing one unity to the second coordinate of the points of $\mathcal{P}_{X^{1}}$. We must go on with the resolution process, because the origin is not an elementary point. It is clear that, at each step, the boundaries of the Newton polygons are the same and hence we apply blowing-ups with same weights to both $(X^{i}_{\Omega},\delta_{i})$ and $(X^{i},\delta_{i})$.

Essentially, we are only desingularizing equilibrium points of $X$ and $X_{\Omega}$ and preserving the impasse curve.

In the end of the process, we have three cases to consider at the impasse point $(0,0)\in\mathbb{E}_{x}$. The first two cases concern the case where the impasse set $\Delta$ does no coincide with a separatrix of the equilibrium point, that is, the Newton polygon $\mathcal{P}_{X}$ has a point of the form $(M,-1)$. The third case concerns the case where $\mathcal{P}_{X}$ does not have a point of the form $(M,-1)$ (which means that the impasse set coincides with a separatrix of the equilibrium point).

\textbf{Case 1:} The Newton polygons $\mathcal{P}_{X_{\Omega, A}}$ and $\mathcal{P}_{X_{A}}$ do not have a vertex of the form $(m, 1)$. This means that the Newton polygon $\mathcal{P}_{X}$ of $X$ does not have a vertex of the form $(m, 0)$. By the coordinate system adopted in the hypothesis, we know that $(M, -1)$ is a vertex of $\mathcal{P}_{X}$. After the last weighted blow up in the $x$ direction at the origin, it follows by Lemma \ref{lemma-case-2-brunella} that the Newton polygon of $\mathcal{P}_{\widetilde{X}}$ of the adjoint vector field $\widetilde{X}$ will contain the point $(0,-1)$, which means that the origin is not an equilibrium point of $\widetilde{X}$. Concerning equilibrium points that may appear on the divisor, they will be the same for both $(\widetilde{X}_{\Omega},\widetilde{\delta})$ and $(\widetilde{X},\widetilde{\delta})$ and the arrangement of such equilibria are the same for both pairs.

Moreover, when we look to $\mathcal{P}_{\widetilde{X}_{\Omega, A}}$ and $\mathcal{P}_{\widetilde{X}_{A}}$, we see that these polygons have the point $(0,0)$, because the origin is an impasse point (see Figure \ref{fig-brunellamiari-newton-polygon-case-1}). In brief, the origin is not an equilibrium point and it is an impasse point, so the origin is elementary. Observe that we applied the same number of blowing-ups with the same weight at each step in $(X_{\Omega},\delta)$ and $(X,\delta)$. Therefore, both pairs have the same elementary singularity scheme.

\begin{figure}[h!]
  \center{\includegraphics[width=0.50\textwidth]{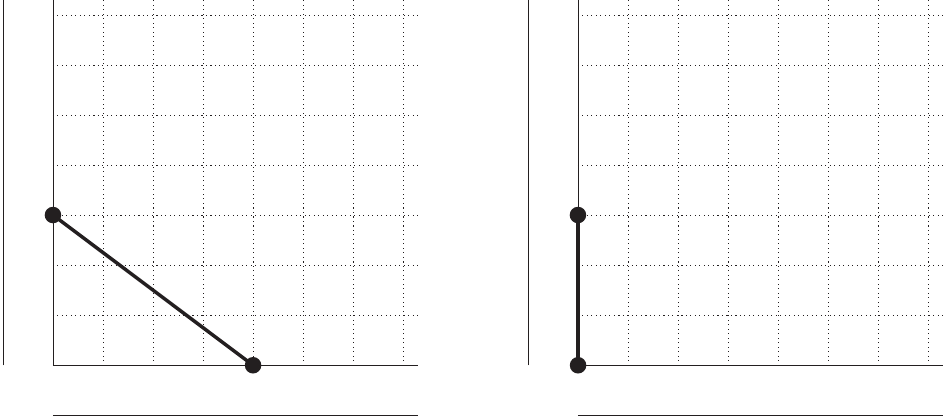}}\\
  \caption{Main segment before and after the last blowing-up in the Case 1, where the Newton polygons $\mathcal{P}_{X_{\Omega, A}}$ and $\mathcal{P}_{X_{A}}$ do not have a vertex of the form $(m, 1)$.}\label{fig-brunellamiari-newton-polygon-case-1}
\end{figure}

\textbf{Case 2:} The Newton polygons $\mathcal{P}_{X_{\Omega, A}}$ and $\mathcal{P}_{X_{A}}$ have a vertex of the form $(m, 1)$. This means that the Newton polygon $\mathcal{P}_{X}$ of $X$ has a vertex of the form $(m, 0)$. By the coordinate system adopted in the hypothesis, we know that $(M, -1)$ is also a vertex of $\mathcal{P}_{X}$, with $m < M$. After a finite number of weighted blow ups in the $x$ direction, the Newton polygon $\mathcal{P}_{\widetilde{X}}$ of the adjoint vector field $\widetilde{X}$ will have the point $(0,0)$, which means that the origin is a semi hyperbolic equilibrium point of $\widetilde{X}$. On the other hand, when we look at $\mathcal{P}_{\widetilde{X}_{\Omega, A}}$ and $\mathcal{P}_{\widetilde{X}_{A}}$, we see that these polygons have the point $(0,1)$, given that the origin is a semi hyperbolic equilibrium point and an impasse point at the same time.

So, this case is different from Case 1 because we need to apply one more blowing up in the objects $(\widetilde{X}_{\Omega},\delta)$ and $(\widetilde{X},\delta)$ in the $x$ direction. Observe that, since both $\mathcal{P}_{\widetilde{X}_{\Omega, A}}$ and $\mathcal{P}_{\widetilde{X}_{A}}$ have points of the form $(0,1)$ and $(\widehat{M},0)$, we will apply the same weighted blow up to both $(\widetilde{X}_{\Omega},\delta)$ and $(\widetilde{X},\delta)$. At the end of the process, by Lemma \ref{lemma-case-2-brunella} we see that $(X_{\Omega},\delta)$ and $(X,\delta)$ have the same elementary singularity scheme.

\begin{figure}[h!]
  \center{\includegraphics[width=0.50\textwidth]{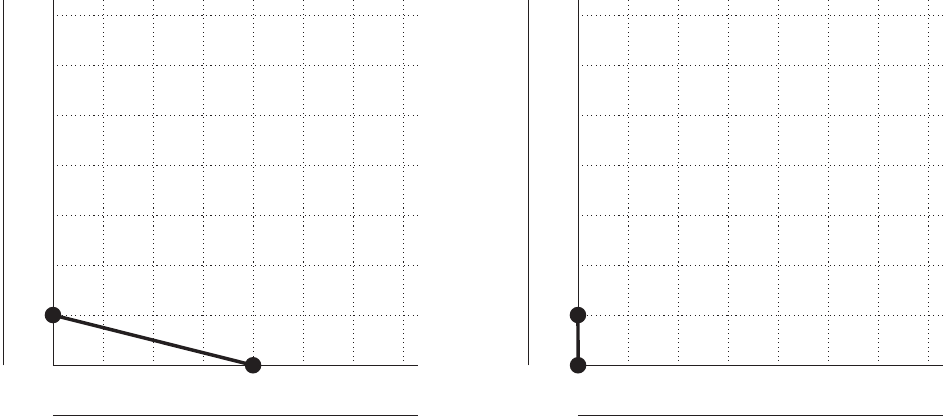}}\\
  \caption{Main segment before and after the last blowing-up in the Case 2, where the Newton polygons $\mathcal{P}_{X_{\Omega, A}}$ and $\mathcal{P}_{X_{A}}$ have a vertex of the form $(m, 1)$.}\label{fig-brunellamiari-newton-polygon-case-2}
\end{figure}

\textbf{Case 3:} Suppose that the impasse set coincides with a separatrix of $(0,0)$. Following the ideas presented in the previous cases, one sees that the process of desingularization of the constrained system is, essentially, the process of the desingularization of the adjoint vector field. After a suitable sequence of weighted blow-ups, the equilibrium points in the exceptional divisor are the same for both vector fields $X_{\Omega}$ and $X$. Furthermore, the arrangement of such equilibrium points are the same for $X_{\Omega}$ and $X$ and they are elementary. Since the impasse set coincides with a separatrix of $(0,0)\in\mathbb{E}_{x}$, it follows that $(X_{\Omega},\delta)$ and $(X,\delta)$ have the same elementary singularity scheme.
\end{proof}

Combining Theorem \ref{coroo-eq-conjugacy} and Proposition \ref{teo-brunellamiari}, we obtain the Theorem \ref{teo-brunellamiari-2}.

\subsection{Example}
\noindent

Consider the diagonalized constrained system
\begin{equation}\label{eq-exe-brunella-miari}
y\dot{x} = y^{3} + x^{2}y + x^{4}; \ y\dot{y} = x^{3} + xy^{2} + y^{4};
\end{equation}
which is already written in favorable coordinates. The support of its auxiliary vector field $X_{A}$ is the set $\mathcal{Q} = \{(-1,4),(1,2),(3,1),(3,0),(0,4)\}$, and therefore the principal part of \eqref{eq-exe-brunella-miari} is given by
\begin{equation}\label{eq-exe-brunella-miari-principal-part}
y\dot{x} = y^{3} + x^{2}y; \ y\dot{y} = x^{3} + xy^{2}.
\end{equation}

\begin{figure}[h!]
  \center{\includegraphics[width=0.25\textwidth]{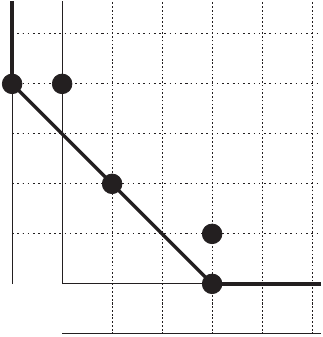}}\\
  \caption{Newton polygon of the auxiliary vector field of \eqref{eq-exe-brunella-miari}.}\label{fig-exe-brunella-miari-newton}
\end{figure}

Observe that the constrained systems \eqref{eq-exe-brunella-miari} and \eqref{eq-exe-brunella-miari-principal-part} are under the hypotheses of the Theorem \ref{teo-brunellamiari-2}. The main segment $\gamma_{1}\subset\partial\mathcal{P}_{X_{A}}$ is contained in the affine line $\{r+s = 6\}$, thus the weight vector is $\omega = (1,1)$. It can be checked that, for both systems \eqref{eq-exe-brunella-miari} and \eqref{eq-exe-brunella-miari-principal-part}, in the positive $y$ direction the points $(\pm1,0)\in\mathbb{E}_{y}$ are hyperbolic saddles of the adjoint vector field. On the other hand, in the positive $x$ direction the points $(\pm1,0)\in\mathbb{E}_{x}$ are hyperbolic saddles and the origin is an impasse point. The systems \eqref{eq-exe-brunella-miari} and \eqref{eq-exe-brunella-miari-principal-part} have the same elementary singularity scheme, and therefore they are orientation preserving $C^{0}$-orbitally equivalents. See the Figure \ref{fig-exe-brunella-miari-blow-up}.

\begin{figure}[h!]
  \center{\includegraphics[width=0.60\textwidth]{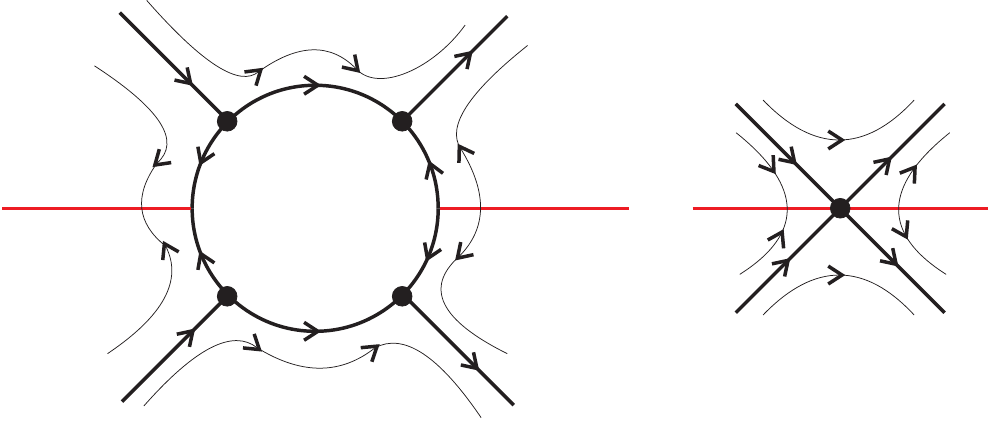}}\\
  \caption{Constrained system \eqref{eq-exe-brunella-miari} and its strict transformed.}\label{fig-exe-brunella-miari-blow-up}
\end{figure}

\section{Acknowledgments}

The first author is supported by Sao Paulo Research Foundation (FAPESP) (grants 2016/22310-0 and 2018/24692-2), and by Coordena\c c\~ao de Aperfei\c coamento de Pessoal de N\'ivel Superior - Brasil (CAPES) - Finance Code 001. The second author was financed by CAPES-Print and FAPESP. The authors are grateful for Daniel Cantergiani Panazzolo for many discussions and suggestions, and for LMIA-Universit\'e de Haute-Alsace for the hospitality during the preparation of this work.





\end{document}